\documentclass{amsart}
\usepackage{amsthm,amsfonts,amsmath,amscd,amssymb,latexsym,epsfig,texdraw}
\newcommand{\gl}{{\mathfrak g \mathfrak l}}

\newcommand{\g}{{\mathfrak g}}         

\newcommand{\cx}{{\mathbb C}}
\newcommand{\diag}{\operatorname{diag}}

\newcommand{\tr}{\operatorname{tr}}

\newcommand{\im}{\operatorname{Im}}

\newcommand{\Lie}{\operatorname{Lie}}
\newcommand{\Hom}{\operatorname{Hom}}

\newcommand{\Aut}{\operatorname{Aut}}
\newcommand{\TAut}{\operatorname{TAut}}
\newcommand{\Ker}{\operatorname{Ker}}

\newcommand{\supp}{\operatorname{supp}}
\newcommand{\Spec}{\operatorname{Spec}}
\newcommand{\Span}{\operatorname{span}}

\newcommand{\wh}{\widehat}
\newcommand{\ol}{\overline}

\numberwithin{equation}{section}

\newtheorem{theorem}{Theorem}[section]

\newtheorem{theorema}{Theorem}\newtheorem{theorema-cor}{Corollary}

\newtheorem{lemma}[theorem]{Lemma}

\newtheorem{corollary}[theorem]{Corollary}

\newtheorem{proposition}[theorem]{Proposition}

\theoremstyle{remark}

\newtheorem{remark}[theorem]{Remark}

\newtheorem{definition}[theorem]{Definition}

\newtheorem{ack}{Acknowledgment}

\newcommand{\R}{{\mathbb{R}}}

\newcommand{\oN}{{\mathbb{N}}}

\newcommand{\oP}{{\mathbb{P}}}

\newcommand{\oR}{{\mathbb{R}}}

\newcommand{\oZ}{{\mathbb{Z}}}

\newcommand{\sA}{{\mathcal{A}}}   

\newcommand{\sG}{{\mathcal{G}}}   

\newcommand{\sK}{{\mathcal{K}}}
\newcommand{\sL}{{\mathcal{L}}}   

\newcommand{\sR}{{\mathcal{R}}}
\newcommand{\sS}{{\mathcal{S}}}

\newcommand{\fG}{{\mathfrak{g}}}
\newcommand{\fH}{{\mathfrak{h}}}

\newcommand{\fL}{{\mathfrak{l}}}
\newcommand{\fM}{{\mathfrak{m}}}

\newcommand{\fS}{{\mathfrak{s}}}

\begin{document}

\title{On the symplectic structure of instanton moduli spaces}
\author{Roger Bielawski \and Victor Pidstrygach}
\address{}


\maketitle

\thispagestyle{empty}

The subject of this paper is the following family of complex affine varieties:
\begin{equation*} N_{k,\tau}=\left\{(A,B,i,j)\in M_{k,k}\times M_{k,k}\times M_{k,2}\times M_{2,k}; \enskip [A,B]-ij=\tau\cdot 1\right\}/GL(k,\cx),\label{inst}\end{equation*}
where $M_{k,l}$ denotes the space of complex $k\times l$-matrices and $\tau\in \cx$. The group $GL(k,\cx)$ acts by conjugation on $A,B$ and by left (resp. right) translations on $i$ (resp. $j$).
\par
It is well known that these varieties are related to the moduli spaces of (framed) $SU(2)$-instantons, i.e. anti-self-dual $SU(2)$-connections on $\oR^4$. More precisely, $N_{k,0}$ is the moduli space of {\em ideal} instantons, with the real instantons corresponding to the smooth locus of $N_{k,0}$. For $\tau\neq 0$, it has been shown by Nekrasov and Schwartz \cite{NS} that $N_{k,\tau}$ is the moduli space of instantons on a non-commutative $\oR^4$.
\par
The variety $N_{k,\tau}$ is a complex-symplectic quotient of a flat space, and, so, it carries a natural complex-symplectic structure. From a more general point of view, $N_{k,\tau}$ is an example of a quiver variety in the sense of Nakajima \cite{Nak}. In fact, using a trick of Crawley-Boevey \cite{C-B}, we can identify $N_{k,\tau}$ as a subvariety in the moduli space of representations in $(\cx^k,\cx)$ of the double of the following quiver:
\begin{equation}
\begin{texdraw}
\drawdim mm
\arrowheadsize l:1 w:1
\arrowheadtype t:H

\lcir r:2
\rmove (-2 0)
\ravec (0 -0.5)
\rmove (4 0.5)
\fcir f:0 r:0.7
\rmove (5.2 0)
\lellip rx:5 ry:1
\rmove (0 1)
\ravec (1 0)
\rmove (-1 -2)
\ravec (1 0)
\rmove (4.4 1)
\fcir f:0 r:0.7
\end{texdraw}
\label{Q1}
\end{equation}

Our starting point is the following trivial observation: the double of the quiver \eqref{Q1} coincides with the double of the following quiver $Q$:

\begin{equation}
\begin{texdraw}
\drawdim mm
\arrowheadsize l:1 w:1
\arrowheadtype t:H

\lcir r:2
\rmove (-2 0)
\ravec (0 -0.5)
\rmove (4 0.5)
\fcir f:0 r:0.7
\rmove (5.2 0)
\lellip rx:5 ry:1
\rmove (0 1)
\ravec (-1 0)
\rmove (1 -2)
\ravec (1 0)
\rmove (4.4 1)
\fcir f:0 r:0.7

\end{texdraw}
\label{Q2}
\end{equation}

This point of view turns out to be very useful. For example, we show:
\begin{theorema} \begin{itemize}
\item[(i)] The space $N_{k,\tau}$ is algebraically completely integrable.
\item[(ii)]  The Poisson algebra $\cx\bigl[N_{k,\tau}\times\cx^2]$ is isomorphic, up to localisation, to the Poisson algebra  $\bigl(\cx\bigl[T^\ast\cx^k\bigr]\bigr)^{S_k} \otimes \bigl(\cx\bigl[T^\ast\cx^{k+1}\bigr]\bigr)^{S_{k+1}}$.
\end{itemize}\end{theorema}
Here ``algebraic complete integrability" is meant in the generalised sense, i.e. we do not require the common level sets of commuting Hamiltonians to be compact.
\par
Part (i) of the above theorem is known: it has been proved by Gibbons and Hermsen in \cite{GH} and it also follows from general results about quiver varieties. Indeed, let  $\sR(Q,V)\simeq M_{k,k}\times M_{k,1}\times M_{1,k}$ denote the space of representations of $Q$ in $V=(\cx^k,\cx)$. Let $Q^{\rm op}$ be a quiver obtained by reversing all arrows in \eqref{Q2}. Then $N_{k,\tau}$ is a subvariety in $\bigl(\sR(Q,V)\times\sR(Q^{\rm op},V)\bigr)/GL(k,\cx)$ and the Poisson algebra $\cx\bigl[N_{k,\tau}]$ contains two commutative Poisson subalgebras $\cx[\sR(Q,V)]^{GL(k,\cx)}$  and $\cx[\sR(Q^{\rm op},V)]^{GL(k,\cx)}$. It follows from a result of Bocklandt \cite{Bock} (see also \S 3 below) that the algebro-geometric quotient $\sR(Q,V)/GL(k,\cx)$ is smooth and hence the generators of $\cx[\sR(Q,V)]^{GL(k,\cx)}$ provide a sufficient number of independent commuting Hamiltonians.

\medskip

Our proof of part (i) identifies generators $\cx[\sR(Q,V)]^{GL(k,\cx)}$  different from those of Gibbons and Hermsen. In particular, we are able  to find (algebraic) Darboux coordinates on a Zariski-open subset of $N_{k,\tau}$, from which part (ii) follows. The quadratic Hamiltonian of Gibbons and Hermsen becomes in these coordinates a Hamiltonian on the phase space of indistinguishable particles of two types (see \S \ref{Ham}).

 We remark that Theorem 1  and other results obtained in the next three sections are very reminiscent of known results for the Calogero-Moser spaces \cite{KKS,W}, and we hope that this analogy can be pushed further. 

We then proceed to study the group of symplectomorphisms of $N_{k,\tau}$ and the Lie algebra of its Hamiltonian vector fields via  the non-commutative symplectic geometry. As for any quiver variety, the group of symplectomorphisms of $N_{k,\tau}$ contains  a homomorphic image of the group $\sS\sG$ of non-commutative symplectomorphisms of the path algebra $\cx \ol Q$ of the double of \eqref{Q2} (cf. \cite{Gin}). Similarly, the Lie algebra of  Hamiltonian vector fields on $N_{k,\tau}$ contains a homomorphic image of the so-called necklace algebra $\sL Q$ (cf. \cite{Gin,LBB}). We observe that the necklace algebra for the quiver \eqref{Q2} should be viewed as a non-commutative analogue of the Poisson algebra of  polynomial functions on $\cx^2\times \gl(2,\cx)$. We also discuss the structure of $\sS\sG$; in particular, we show that its Lie algebra ($\sS\sG$ is an algebraic ind-group) is strictly smaller than the Lie algebra  of symplectic derivations of $\cx \ol Q$. The main result in the second part is the following analogue of a theorem of Berest and Wilson \cite{BW}:
\begin{theorema} The action of $\sS\sG$ on $N_{k,\tau}$ is transitive, if $\tau\neq 0$.\end{theorema}
In fact, we show transitivity  of the (apriori strictly smaller) subgroup generated by  automorphisms of $\cx \ol Q$, which preserve either $\cx Q$ or $\cx Q^{\rm op}$. An immediate consequence of this is
\begin{theorema} The Poisson algebra $\cx\bigl[N_{k,\tau}]$ is generated by its two commutative Poisson subalgebras\, $\cx[\sR(Q,V)]^{GL(k,\cx)}$  and\, $\cx[\sR(Q^{\rm op},V)]^{GL(k,\cx)}$.\label{generators}\end{theorema}

\section{Matrix interpretation of the quiver $Q$\label{alt}}

Recall, from the introduction,  the definition of the variety $N_{k,\tau}$. It should be interpreted as saying that $N_{k,\tau}$ is a complex-symplectic quotient of $T^\ast R_k$ by $GL(k,\cx)$, where 
$R_k=M_{k,k}\times M_{2,k}$. The moment map $\mu$ is
\begin{equation} \bigl(M_{k,k}\times M_{2,k}\bigr)\times \bigl(M_{k,k}\times M_{k,2}\bigr)\ni \bigl((A,j),(B,i)\bigr)\stackrel{\mu}{\longmapsto} [A,B]-ij.\label{moment}\end{equation}
Thus, $N_{k,\tau}=\mu^{-1}(\tau\cdot 1)/GL(k,\cx)$.
For $\tau\neq 0$, the action of $PGL(k,\cx)$ on $\mu^{-1}(\tau\cdot 1)$ is free, and all orbits are closed. Thus, $N_{k,\tau}$ is a manifold. For $\tau=0$, the quotient  should be understood as the affine-geometric quotient, i.e.
\begin{equation} N_{k,0}=\Spec \cx\bigl[ \mu^{-1}(0)\bigr]^{GL(k,\cx)}.\label{N_k_0} \end{equation}
In other words $N_{k,0}$ consists of closed orbits of the $GL(k,\cx)$-action on $\mu^{-1}(0)$.
Its smooth locus is isomorphic to the {\em moduli space of framed instantons of charge $k$}.


We now rephrase the definition of $N_{k,\tau}$, corresponding to the passage from \eqref{Q1} to \eqref{Q2}.

Let $(A,B,i,j)\in T^\ast R_k$, i.e. $A,B\in \gl(k,\cx)$, $i\in M_{k,2}$, $j\in M_{2,k}$. Let us write $i=(i_1,i_2)$ and $j=\begin{pmatrix} j_1 \\ j_2 \end{pmatrix}$, and define:
\begin{equation} \widehat{A}=\begin{pmatrix} A&i_1 \\ j_2& 0 \end{pmatrix}, \qquad
\widehat{B}= \begin{pmatrix} B  &i_2 \\ - j_1& 0 \end{pmatrix}.
\label{hats}
\end{equation}
The matrices $\wh{A},\wh{B}$ lie in $\gl(k+1,\cx)$ and can be wiewed as an element of $T^\ast \fS_{k+1}$, where $\fS_{k+1}$ consists of $(k+1)\times (k+1)$-matrices with zero $(k+1,k+1)$-entry.  We consider a subgroup $G\simeq GL(k,\cx)$ of  $GL(k+1,\cx)$, consisting of matrices of the form
$$\left(\begin{array}{cc}  L &  0 \\  0 & 1 \end{array}\right),$$
with $L\in GL(k,\cx)$. Its action on $T^\ast \fS_{k+1}$ is Hamiltonian, with the moment map $\mu_G$ given by the upper-left $k\times k$-minor of the commutator $[\wh{A},\wh{B}]$.
\par
A short calculation shows
\begin{equation}[\widehat{A},\widehat{B}] 
= \begin{pmatrix} [A,B] -ij&Ai_2 - B i_1  \\ j_2B + j_1A& j_1i_1+j_2i_2 \end{pmatrix},
\label{bracket}\end{equation}
and, hence:
\begin{proposition} The map $(A,B,i,j)\mapsto \bigl(\wh{A},\wh{B}\bigr)$ is  equivariant with respect to the actions of $GL(k,\cx)$ on $T^\ast R_k$ and $G$ on $T^\ast \fS_{k+1}$, and it induces a symplectic isomorphism between $N_{k,\tau}$ and the symplectic quotient $\mu_G^{-1}(\tau\cdot 1)/G$ of $T^\ast \fS_{k+1}$ by $G$.\label{iso1}\end{proposition}
\begin{proof} It only remains to check that that the symplectic forms agree, which follows from: $\tr d\wh{A}\wedge d\wh{B}=\tr dA\wedge dB+\tr dj\wedge di.$.\end{proof}
\begin{remark} Another interpretation of this is to say that we have chosen a $GL(k,\cx)$-invariant Lagrangian subspace in $T^\ast R_k$, different from the zero-section.\end{remark}

\section{An embedding $N_{k,\tau}\hookrightarrow N_{k+1,\tau}$}

This section is not used in the remainder of the paper. We wish to describe a rather surprising fact, namely the existence of a symplectic embedding from the $k$-instanton moduli space to the $(k+1)$-instanton moduli space, as well as an analogous map for torsion free sheaves on $\oP^2$.
\par
In the setup of the previous section, consider $(A,B,i,j)\in T^\ast R_k$ satisfying $[A,B]=ij+\tau\cdot 1$ and define the matrices $\wh{A},\wh{B}$ via \eqref{hats}. Formula \eqref{bracket} yields now:
\begin{equation*}
[\widehat{A},\widehat{B}]-\tau\cdot 1
= \begin{pmatrix} \tau\cdot 1 &Ai_2 - B i_1 \\ j_2B + j_1A& -k\tau \end{pmatrix}.
\label{hat_bracket}\end{equation*}
Thus, $[\widehat{A},\widehat{B}]-\tau\cdot 1$   is also of rank $2$, and we can write
$$[\widehat{A},\widehat{B}]=\hat{\imath}\hat{\jmath},$$
where $\hat{\imath}=(\hat{\imath}_1,\hat{\imath}_2)$, $\hat{\jmath}=\begin{pmatrix}\hat{\jmath}_1,\hat{\jmath}_2\end{pmatrix}$ with
$$ \hat{\imath}_1=(0,\dots,0,1)^T,\enskip \hat{\imath}_2=(Ai_2 - B i_1,-(k+1)\tau)^T,\enskip \hat{\jmath}_1=(j_2B + j_1A,0),\enskip  \hat{\jmath}_2=(0,\dots,0,1).$$
The assignment 
\begin{equation} (A,B,i,j)\longmapsto \bigl(\widehat{A},\widehat{B}, \hat{\imath},\hat{\jmath}\bigr)\label{embedd}\end{equation}
is $GL(k,\cx)$-equivariant and, hence, it induces a map $N_{k,\tau}\rightarrow N_{k+1,\tau}$. 
First of all, we have
\begin{proposition} The map $N_{k,\tau}\rightarrow N_{k+1,\tau}$ defined by \eqref{embedd} is a symplectic embedding.\label{emb1}\end{proposition}
\begin{proof} The map is an embedding, since the subgroup of $GL(k+1,\cx)$ preserving $\hat \imath_1=(0,\dots,0,1)^T$ and $\hat \jmath_2=(0,\dots,0,1)$ is $G$.
The fact that the map respects the symplectic forms follows immediately from the following calculation:
$$ \tr d\wh{A}\wedge d\wh{B} +\tr d\hat{\jmath}\wedge d\hat{\imath}=\tr d\wh{A}\wedge d\wh{B}=\tr dA\wedge dB+\tr dj\wedge di.$$\end{proof}

The image of this embedding is easily seen to consist of $GL(k+1,\cx)$-equivalence classes of $(C,D,x,y)$ such that
$$ C_{k+1,k+1}= D_{k+1,k+1}=0,\quad  yx=\begin{pmatrix} 0 & \ast\\ 1 & -(k+1)\tau\end{pmatrix}.$$

Even more remarkable is the fact that, for $\tau=0$, this map respects various stability conditions. We recall that a quadruple $(A,B,i,j)\in T^\ast R_k$ is called
\begin{itemize}
\item  {\em stable}, if there is no proper subspace $V$ of $\cx^k$ with $A(V)\subset V$, $ B(V)\subset V$ and $\im i\subset V$,
\item {\em co-stable}, if there is no proper subspace $V$ of $\cx^k$ with $A(V)\subset V$, $ B(V)\subset V$ and $V \subset \Ker j$,
\item {\em regular}, if it is both stable and co-stable.
\end{itemize}

A result of Nakajima \cite{Nak0, Nak} states that 
$$\tilde N_{k,0}=\left\{(A,B,i,j)\in T^\ast R_k;\enskip [A,B]=ij, \enskip\text{$ (A,B,i,j)$ is stable}\right\}/GL(k,\cx),$$
is a resolution of $N_{k,0}$ and it can be identified with the framed moduli space of torsion free sheaves on $\oP^2$ with rank $2$, $c_1=0$, and $c_2=k$.
\par
Similarly, the smooth locus $N_{k,0}^{s}$ of $N_{k,0}$ can be identified with the set of $GL(k,\cx)$-equivalence classes of regular  $(A,B,i,j)$ satisfying $[A,B]=ij$. The argument follows  the proof of Lemma 3.25 in \cite{Nak}. In turn $N_{k,0}^{s}$ can be identified with the moduli space of $SU(2)$-instantons on $\oR^4$ or with the framed moduli space of rank $2$ vector bundles on $\oP^2$ with $c_1=0$ and $c_2=k$.
\par
We have
\begin{proposition} The map \eqref{embedd} preserves the conditions of stability, co-stability and regularity.\end{proposition}
\begin{proof}
Suppose that $V\subset \cx^{k+1}$ is a linear subspace such that $\wh{A}(V)\subset V$, $\wh{B}(V)\subset V$ and $\im (\hat{\imath})\subset V$. This last condition implies that $V_{k+1}=\cx(0,\dots,0,1)^T\subset V$ and, hence, if $(v_1,\dots,v_{k+1})^T\in V$, then $(v_1,\dots,v_{k},0)^T\in V$. Therefore $V=W\oplus V_{k+1}$, where $W$ is the projection of $V$ onto the subspace  $\{(v_1,\dots,v_{k},0)\in \cx^{k+1}\}$. The assumptions $\wh{A}(V)\subset V$ and $\wh{B}(V)\subset V$, applied to $V_{k+1}$, show that $\im (i)\subset W$. Using this and again $\wh{A}(V)\subset V$ and $\wh{B}(V)\subset V$, we conclude that $A(W)\subset W$ and $B(W)\subset W$. Thus stability is preserved. A similar argument shows that co-stability is preserved and, hence, so is the regularity.
\end{proof}

\begin{corollary} The map \eqref{embedd} induces symplectic embeddings $N_{k,0}^{s}\hookrightarrow N_{k+1,0}^{s}$ and $\tilde N_{k,0}\hookrightarrow\tilde N_{k+1,0}$.
\end{corollary}

We do not understand how this map looks like in terms of sheaves, i.e. how do we get a framed torsion-free sheaf with $c_2=k+1$ from one with $c_2=k$.


\section{The action of $G$ on $\gl(k+1,\cx)$}

In this section, we consider in detail ``half" of the set-up described in \S\ref{alt}, i.e. the adjoint action of $G\simeq GL(k,\cx)$ on $\gl(k+1,\cx)$. We shall write $\wh\g=\gl(k+1,\cx)$ and decompose elements of $\wh\g$ as
\begin{equation}\wh A=\begin{pmatrix} A & x \\ y & \Lambda\end{pmatrix},\quad A\in \g,\enskip x\in M_{k,1},\enskip y\in M_{1,k},\enskip \Lambda\in \cx.\label{hatA}\end{equation}

\subsection{A slice to the $G$-action} 
We recall that an element $\wh{A}$ of $\wh{\fG}$ is called {\em $G$-semisimple} if its
$G$-orbit is closed, and is called {\em $G$-regular}, if its $G$-orbit has maximal
dimension. We denote by $\wh{\fG}^{\rm r}$ the subset of all $G$-regular points of $\wh\fG$.

Let $\wh{A}$ be a $G$-regular element of the form \eqref{hatA}. Then $A$ is regular in $\gl(k,\cx)$, and it can be conjugated to a matrix of the form:
\begin{equation}\left(
\begin{array}{ccccc}
0&\cdots & 0&0& r_1 \\
1&\cdots & 0&0&r_2 \\
\vdots&\ddots &\vdots&\vdots&\vdots\\
0&\cdots& 1& 0 & r_{k-1} \\
0&\cdots& 0& 1& r_k \\ \end{array} \right).\label{Sl1}\end{equation}
We would also like to put the covector $y$ in a standard form, say $(0,\dots,0,1)$. We have:
\begin{lemma}  Let $A$ be a matrix of the form \ref{Sl1} and let $y=(y_1,\dots,y_k)$ be a covector. There exists an invertible matrix $X$ such that  $XAX^{-1}=A$ and $yX^{-1}=(0,\dots,0,1)$ if and only if $yv\neq 0$ for any eigenvector $v$ of $A$. If such an $X$ exists, then it is unique.\label{X}\end{lemma}
\begin{proof} Since $(0,\dots,0,1)$ is a cyclic covector for $A$, there exists a unique $X=\sum_{i=0}^{k-1} c_iA^i$ such that $y=(0,\dots,0,1)X$. We know that $[X,A]=0$ and the problem is the invertibility of $X$.  If we put $A$ in the Jordan form, then it is clear that $\det X\neq 0$ if and only if  $\sum_{i=0}^{k-1} c_i\lambda^i\neq 0$ for any eigenvalue $\lambda$ of $A$. Let $v=(v_1,\dots,v_k)^T$ be an eigenvector for $A$ with the eigenvalue $\lambda$. We observe that $Av=\lambda v$ and $v\neq 0$ implies that $v_k\neq 0$. Since $yv=(0,\dots,0,1)Xv=v_k\sum_{i=0}^{k-1} c_i\lambda^i$, we conclude that $\det X\neq 0$ precisely when $yv\neq 0$ for any eigenvector $v$.
\end{proof}

We observe that the condition $yv\neq 0$ for any eigenvector $v$ of $A$ is equivalent to $A$ and $\wh{A}$ not having a common eigenvector with a common eigenvalue. We therefore define the following set:
\begin{equation}\wh{\fG}^{0}=\left\{\wh{A}\in \wh{\fG}^{\text{r}};\forall_v \enskip Av=\lambda v\implies \wh{A}\begin{pmatrix}v\\0\end{pmatrix}\neq \lambda \begin{pmatrix}v\\0\end{pmatrix}\right\}.\label{S0}\end{equation}

We conclude, from the last lemma, that any element of $\wh{\fG}^{0}$ is $G$-conjugate to a matrix of the form:
\begin{equation}\left(
\begin{array}{cccccc|c}
0&0&\cdots & 0&0& r_1 &s_1\\
1&0&\cdots & 0&0&r_2 & s_2\\
\vdots&\ddots&\ddots &\vdots&\vdots&\vdots&\vdots\\
\vdots&\vdots&\ddots &\ddots&\vdots&\vdots&\vdots\\
0&0&\cdots& 1& 0 & r_{k-1} & s_{k-1}\\
0& 0&\cdots& 0& 1& r_k & s_k\\
\hline 0&0&\cdots&0& 0& 1&  s_{k+1}
\end{array} \right).\label{Slodowy}\end{equation}
Let $\sS$ be
the set of matrices of this form. We rephrase Lemma \ref{X} as follows:
\begin{theorem} The set $\sS$ meets any $G$-orbit in $\wh{\fG}^0$ in exactly one
point.\hfill $\Box$\label{slice}\end{theorem}
We have an immediate corollary:
\begin{corollary} Any element of $\wh{\fG}^0$ is regular (as a
matrix).\hfill $\Box$\label{regular}\end{corollary}

\subsection{$G$-invariants}
Since $\wh{\fG}^0$  is Zariski-open, any polynomial invariant of the $G$-action on $\wh{\fG}$ is an algebraic function on the slice $\sS$, i.e. a polynomial in $r_i,s_j$. The functions $r_i$ and $s_j$ are, in turn, given by the characteristic polynomials of $A$ and $\wh{A}$:
\begin{proposition} Let $\wh{A}=A+M$ be an element of $\wh{\fG}^0$ of the form
\eqref{Slodowy}, and let $q(z),\wh{q}(z)$ be the characteristic polynomials of $A$ and $\wh{A}$. Then:
\begin{equation} z^k-\sum_{i=1}^k r_iz^{i-1}=q(z),\label{r} \end{equation}
\begin{equation} \sum_{i=1}^{k} s_i
z^{i-1}=(z-s_{k+1})q(z)-\wh{q}(z).\label{s}\end{equation}
\label{coeff}\end{proposition}
\begin{proof} The matrix \eqref{Slodowy} represents the multiplication by $z$ on
$\cx[z]/(\wh{q})$ in the basis $1,z,\dots ,z^{k-1},q(z)$.\end{proof}

 The existence of the slice $\sS$ and Proposition \ref{coeff} imply the following description of the ring of $G$-invariant functions on $\wh{\fG}$:
\begin{corollary} $\cx[\wh{\fG}]^G=\cx[\fH_k]^{S_k}\otimes \cx[\fH_{k+1}]^{S_{k+1}}$, where $\fH_k$, ${\fH}_{k+1}$ denote  Cartan subalgebras of $\gl(k,\cx),\gl(k+1,\cx)$. \hfill $\Box$ \label{invariants}\end{corollary}

\subsection{Strongly semisimple points}
\begin{definition} A matrix $\wh{A}$ of the form \eqref{hatA} is said to be {\em strongly semisimple}, if $\wh{A}\in\wh{\fG}^0$ and both $A,\wh{A}$ are regular semisimple matrices. The subset of strongly semisimple points of $\wh{\fG}$ will be denoted by $\wh{\fG}^1$.\label{ss}\end{definition}
If $A$ is regular semisimple, then $\wh{A}$ can be $G$-conjugated  to a matrix of the form:
\begin{equation} \left(
\begin{array}{ccc|c}
\lambda_1&\cdots&0&x_1\\
\vdots&\ddots&\vdots&\vdots\\
0&\cdots&\lambda_k&x_k\\
\hline y_1&\cdots&y_k& \Lambda
\end{array} \right),\label{canonical0}
\end{equation}
where the $\lambda_i$ are distinct.
We observe that an $\wh{A}$ of this form is $G$-regular and $G$-semisimple precisely when
$x_iy_i\neq 0$ for all $i$. Similarly, $\wh{A}$ is $G$-regular (resp. in $\wh{\fG}^0$) if
and only if, for any $i$,  either $y_i$ or $x_i$ is nonzero (resp. $y_i\neq 0$). In particular, any element of $\wh{\fG}^0$ with semisimple $A$ can be $G$-conjugated to
\begin{equation} \left(
\begin{array}{ccc|c}
\lambda_1&\cdots&0&x_1\\
\vdots&\ddots&\vdots&\vdots\\
0&\cdots&\lambda_k&x_k\\
\hline 1&\cdots & 1& \Lambda
\end{array} \right),\label{canonical}
\end{equation}

Suppose now that $\wh{A}$ of this form is strongly semisimple, with (distinct) eigenvalues $\wh{\lambda}_1,\dots,\wh{\lambda}_{k+1}$.  Let $q(z)=\prod (z-\lambda_i)$ and $\wh{q}(z)=\prod (z-\wh{\lambda}_j)$ be the characteristic polynomials of $A$ and $\wh{A}$.
We consider the multiplication by $z$ on $\cx[z]/(\wh{q})$. It is a linear
operator which in the basis
\begin{equation} \prod_{j\neq i}(z-\wh{\lambda}_j), \quad i=1,\dots,k+1
\label{base1}\end{equation}
is the diagonal matrix $\diag(\wh{\lambda}_1,\dots,\wh{\lambda}_{k+1})$. On the
other hand, in the basis
\begin{equation} \prod_{j\neq i}(z-\lambda_j),q(z), \quad i=1,\dots,k,
\label{base2}\end{equation}
multiplication by $z$ is given by the matrix of the form \eqref{canonical}. The numbers
$x_i$ satisfy
$$ (z-\Lambda)q(z)=\sum_i\left(x_i \prod_{j\neq i}(z-\lambda_j)\right)\quad \mod\wh{q},$$
and hence
$$(z-\Lambda)q(z)-\wh{q}(z)=\sum_i\left(x_i \prod_{j\neq i}(z-\lambda_j)\right).$$
Substituting $\lambda_i$ for $z$ we obtain:
\begin{equation} x_i=-\frac{ \prod_{j=1}^{k+1} (\lambda_i-\wh{\lambda}_j)}{
\prod_{j\neq i} (\lambda_i-\lambda_j)}.\label{x_i}\end{equation}
From this formula and the remarks after \eqref{canonical0}, we obtain immediately:
\begin{corollary} Let $A$ be regular semisimple. Then $\wh{A}$ is $G$-regular and $G$-semisimple if and only if $A$ and $\wh{A}$ do not have a common eigenvalue. \hfill $\Box$.\label{rss}\end{corollary}
Let now $g$ be the matrix representing  the passage from the basis \eqref{base1}
to \eqref{base2}. Then $g$ diagonalises \eqref{canonical}, i.e. $g\wh{A}g^{-1}=\diag \bigl(\wh{\lambda}_{1},\dots, \wh{\lambda}_{k+1}\bigr)$. We compute easily the entries of $g$ and of $g^{-1}$:
\begin{equation} g_{ij}=\begin{cases} \frac{\prod_{m\neq
j}(\wh{\lambda}_i-\lambda_m)}{ \prod_{n\neq i}(\wh{\lambda}_i-\wh{\lambda}_n)} &
\text{if $j\leq k$}\\ & \\
\frac{\prod_{ m=1}^k (\wh{\lambda}_i-\lambda_m)}{ \prod_{n\neq
i}(\wh{\lambda}_i-\wh{\lambda}_n)} & \text{if $j=k+1$},\end{cases}
\label{g_{ij}}\end{equation}

\begin{equation} g^{ij}=[g^{-1}]_{ij}=\begin{cases}\frac{\prod_{m\neq j}(\lambda_i-\wh{\lambda}_m)}{
\prod_{n\neq i}(\lambda_i-\lambda_n)} & \text{if $i\leq k$}\\  & \\ 1 & \text{if $i=k+1$}.\end{cases}
\label{g^{ij}}\end{equation}

\section{On the Poisson structure of the instanton moduli space}

We go back to the space $N_{k,\tau}$, which we have identified, in Proposition \ref{iso1} with the symplectic quotient $\mu_G^{-1}(\tau\cdot 1)$ of $T^\ast {\fS}_{k+1}$ by $G$. Our aim is to embedd the Poisson algebra of functions on $N_{k,\tau}$ into the Poisson algebra of $S_k\times S_{k+1}$-invariant functions on an open subset of a flat space. It turns out that it is better to consider $N_{k,\tau}\times \cx^2$, with the standard symplectic structure on the second factor. In other word, we consider the symplectic quotient $\mu_G^{-1}(\tau\cdot 1)$ of the whole $T^\ast \gl({k+1},\cx)$ by $G$.
\par
For a $\tau\in \cx$, we write
$$\wh{\tau}=\diag\bigl(\tau,\dots,\tau,-k\tau\bigr).$$
We also write $\fM^-$ (resp. $\fM^|$) for the subspace of $\wh{\g}=\gl({k+1},\cx)$ generated by the $y$-s (resp. $x$-s) in \eqref{hatA}, and $\fM=\fM^-\oplus\fM^|$.
\par  
 The equation
$\mu_G(\wh{A},\wh{B})=\tau\cdot 1$ can be written as
\begin{equation} \bigl[\wh{A},\wh{B}]\in \wh{\tau}+\fM.\label{moment2}\end{equation}

\subsection{Decomposition of $\wh{B}$} 
Consider, for now, the case $\tau=0$. The equation \eqref{moment2} reduces then to
\begin{equation} \bigl[\wh{A},\wh{B}]\in \fM.\label{moment3}\end{equation}
We make the following observation 
\begin{lemma} Let $\wh{A}\in\wh{\fG}$ and let $B_1\in\g$, $B_2\in \wh{\fG}$,  be such that
$[A,B_1]=0$, $[\wh{A},B_2]=0$. Then the pair $(\wh{A},B_1+B_2)$ satisfies
\eqref{moment3}.\label{sum}\end{lemma}
\begin{proof} Follows from $[\wh{A},B_1]\in\fM$.\end{proof}

Thus, we would like to ask which $\wh{B}$ satisfying \eqref{moment3} can be decomposed into
$B_1$ and $B_2$ as in the lemma. We have:
\begin{proposition} Let $\wh{A}$ be $G$-regular with both $A$ and $\wh A$ regular semisimple. Then any $\wh{B}$ satisfying
$[\wh{A},\wh{B}]\in \fM$ can be written uniquely as $B_1+B_2$, $B_1\in \g$, with $[A,B_1]=0$,
 $[\wh{A},B_2]=0$.  \label{decompose}\end{proposition}
\begin{proof} The assumption implies that $\wh{A}$ is $G$-conjugate to a matrix of the form \eqref{canonical0}, with  $\lambda_i$ distinct, and that, for any $i\leq k$, either $x_i$ or $y_i$ is nonzero.  Let $H$ be
the vector space of $\wh{B}\in \wh{\fG}$ such that $[\wh{A},\wh{B}]\in \fM$.
 The  set $\wh{H}$  of elements commuting with $\wh{A}$ is a subspace of $H$.  We aim to
show that
$$H=\wh{H}\oplus \bigl(\g\cap H\bigr).$$
This is sufficient, as an element of $\g\cap H$ commutes with $A$. 
\par
To prove existence of the above decomposition, first of all notice that the subspaces on
the right have $0$-intersection. Indeed, suppose to the contrary, that $C\in \wh{H}\cap
\g\cap H=\wh{H}\cap\g$. Then $C$ is a diagonal element of $\g$ and $[\wh{A},C]$ is the
element $yC-Cx$ of $\fM$. As $[\wh{A},C]=0$ and $x_i$ or $y_i$ is nonzero for any $i$,
$C$ must be $0$.
\par
To prove that any element of $H$ has the desired decomposition it is now sufficient to
show that $\dim \g\cap H\geq \dim H-k-1$ (since  $\dim\wh{H}\geq k+1$).
 Let $\pi:\wh{\fG}\rightarrow \fM$ be the orthogonal projection.  We observe that an
element $J+I$ of $\fM$, $J\in \fM^-$, $I\in \fM^|$, which is in $\pi(H)$ must satisfy the following condition: the
diagonal part of $xJ-Iy$ vanishes. Therefore (as either $x_i$ or $y_i\neq 0$ for any $i$)
$\dim \pi(H)\leq k$ and we are done.  
\end{proof}

We would like now a similar decomposition for an arbitrary $\tau$.

\begin{proposition} Let $\wh{A}\in \wh{\fG}^1$. Then any $\wh{B}$ satisfying
$[\wh{A},\wh{B}]\in \wh{\tau}+\fM$ can be written uniquely as $B_1+B_2$, $B_1\in \g$, with $[A,B_1]=0$,
 $[\wh{A},B_2]\in \wh{\tau}+\fM^|$. Moreover the entries of  $[\wh{A},B_2]$ are uniquely determined by $\tau$, the eigenvalues of $A$ and the eigenvalues of $\wh{A}$. \label{decompose2}\end{proposition}
Recall that $\fM^|$ denotes the subset of $\wh{\fG}$ with all entries, apart from those in the last column, equal to zero.
\begin{proof} We can assume that $\wh{A}$ is of the form \eqref{canonical}, with $\lambda_i$ distinct. It is easy to see that there is a unique $B_1$ of the form $\diag(\mu_1,\dots,\mu_k,0)$, such that the $\fM^-$-components of $[\wh{A},B_1]$ and $[\wh{A},\wh{B}]$ coincide. This proves the first part of the proposition. For the second part, let $C_1,C_2$ be two $(k+1)\times (k+1)$-matrices  with $[\wh{A},C_i]\in \wh{\tau}+\fM^|$, $i=1,2$. We need to show that $[\wh{A},C_1-C_2]=0$.  We know that $m=[\wh{A},C_1-C_2]\in\fM^|$, and, therefore,  Proposition \ref{decompose} allows us to write $C_1-C_2=U_1+U_2$, where $U_1$ is diagonal and  $[\wh{A},U_2]=0$. Since all the $\fM^-$-entries of $\wh{A}$ are nonzero, $[\wh{A},U_1]\in\fM^|$ implies that $U_1=0$ and, hence, $m=0$.\end{proof}

\subsection{A symplectomorphism\label{smor}}
We denote by $\tilde{N}_{k,\tau}^{\rm reg}$ the open subset of $N_{k,\tau}\times \cx^2$ formed by $G$-orbits of $\bigl(\wh{A},\wh{B}\bigr)$ such that $\wh{A}\in \wh{\fG}^{1}$. 
Let $\wh{A}$ be of the form \eqref{canonical} and let $\wh{B}$ satisfy \eqref{moment2}. According to the last proposition, we can decompose $\wh{B}$ as $B_1+B_2$ with $B_1\in \g$, $[A,B_1]=0$ and $[\wh{A},B_2]=\wh{\tau}+m$, with $m\in\fM^|$. Therefore $B_1$ is diagonal:
\begin{equation} B_1=\diag(\mu_1,\dots,\mu_k,0).\label{B_1}\end{equation}
On the other hand, let $g$ be the matrix \eqref{g_{ij}} diagonalising  $\wh{A}$, so that $g\wh{A}g^{-1}=D_{\wh{\lambda}}=\diag\bigl(\wh{\lambda}_1,\dots,\wh{\lambda}_{k+1}\bigr)$. Then $\bigl[D_{\wh{\lambda}},gB_2g^{-1}\bigr]=g(\wh{\tau}+m)g^{-1}$. Therefore (using the second part of Proposition \ref{decompose2}), the off-diagonal terms of $gB_2g^{-1}$ are determined by the $\lambda_j$ and $\wh{\lambda}_i$ (and by $\tau$). Hence, 
the diagonal entries $\wh{\mu}_1,\dots,\wh{\mu}_{k+1}$ of $gB_2g^{-1}$ provide additional coordinates and we can write
\begin{equation} B_2=g^{-1}D_{\wh{\mu}}g+g^{-1}S g,\label{gS}\end{equation}
where $S$ is off-diagonal and depends only on $\tau,\lambda_{i},\wh{\lambda}_j$, $i=1,\dots,k$, $j=1,\dots,k+1$.

Let $\fH_k$ and $\fH_{k+1}$ denote Cartan subalgebras of $\gl(k,\cx)$ and $\gl(k+1,\cx)$, respectively, and  write ${\fH_k}^{\rm reg}$, $\fH_{k+1}^{\rm reg}$  for the Zariski open subsets, where the actions of $S_k$ and $S_{k+1}$ are free.  
The assignment
\begin{equation} \bigl(\wh{A},\wh{B}\bigr)\longmapsto \left(\lambda_i,\wh{\lambda}_j,\mu_i,\wh{\mu}_j\right)_{\stackrel{\scriptscriptstyle i=1,\dots,k}{\scriptscriptstyle j=1,\dots,k+1}}
\label{map}\end{equation}
gives a well defined map $\pi:\tilde{N}_{k,\tau}^{\rm reg}\rightarrow \left(T^\ast\fH_k^{\rm reg}\times T^\ast\fH_{k+1}^{\rm reg}\right)/S_k\times S_{k+1}$. 
\begin{remark} The map $\pi$ is equivariant for the action of $\cx^2$, given by $$\wh A\mapsto \wh A+z_1\cdot 1,\enskip \wh B\mapsto \wh B+z_2\cdot 1,\enskip \lambda_i\mapsto \lambda_i+ z_1,\enskip \wh{\lambda}_j\mapsto \wh{\lambda}_j+ z_1,\enskip \mu_i\mapsto\mu_i,\enskip \wh{\mu}_j\mapsto\wh{\mu}_j+ z_2.$$ The slice $N_{k,\tau}$ to this action is defined by $\wh{A}_{k+1,k+1}=0= \wh{B}_{k+1,k+1}$ and this does not map, via $\pi$, to a linear subspace. This is the reason why we consider $N_{k,\tau}\times\cx^2$, instead of $N_{k,\tau}$.\end{remark}
We have:
\begin{proposition} The map $\pi$ is a symplectic isomorphism.\label{symplectomorphism}\end{proposition}
\begin{proof}
We compute:
$$
\tr\bigl(d \widehat{A} \wedge d \widehat{B}\bigr) =\tr\bigl(d\wh{A}\wedge dB_1\bigr)+\tr\bigl(d\wh{A}\wedge dB_2\bigr)=$$
$$   \tr\left(d \widehat{A} \wedge d D_{\mu}\right)  + 
\tr\left(d \widehat{A} \wedge d(g^{-1}(D_{\widehat{\mu}}+S)g)\right) =
$$
$$
=\sum_pd\lambda_p \wedge d \mu_p + \tr\left(dD_{\widehat{\lambda}}^{g^{-1}}) \wedge d\bigl(D_{\widehat{\mu}}+S\bigr)^{g^{-1}}\right).
$$
Now,
$$\tr\left(dD_{\widehat{\lambda}}^{g^{-1}}) \wedge d\bigl(D_{\widehat{\mu}}+S\bigr)^{g^{-1}}\right)
 = d \left( \tr D_{\widehat{\lambda}}^{g^{-1}}  d \bigl(D_{\widehat{\mu}}+S\bigr)^{g^{-1}} \right),
$$

and
$$
 \tr D_{\widehat{\lambda}}^{g^{-1}}  d D_{\widehat{\mu}}^{g^{-1}}= 
 \tr D_{\widehat{\lambda}}^{g^{-1}} \left( (d D_{\widehat{\mu}})^{g^{-1}} - \bigl[g^{-1}dg, D_{\widehat{\mu}}^{g^{-1}}\bigr] \right) = 
$$
$$
 = \tr D_{\widehat{\lambda}} d D_{\widehat{\mu}} - \tr D_{\widehat{\lambda}}^{g^{-1}}\bigl[g^{-1}dg, D_{\widehat{\mu}}^{g^{-1}}\bigr] = 
\tr D_{\widehat{\lambda}}d D_{\widehat{\mu}} + \tr D_{\widehat{\lambda}}[dgg^{-1}, D_{\widehat{\mu}}]
$$
$$
= \tr D_{\widehat{\lambda}}d D_{\widehat{\mu}} +0,
$$
where we used the fact that off-diagonal matrix $[g^{-1}dg,D_{\widehat{\mu}}]$ is perpendicular to the diagonal one. 

We compute the remaining term:
$$
 \tr D_{\widehat{\lambda}}^{g^{-1}}  d S^{g^{-1}}=\tr D_{\widehat{\lambda}}^{g^{-1}} \left( (dS)^{g^{-1}} - \bigl[g^{-1}dg, S^{g^{-1}}\bigr] \right) = 0-\tr D_{\widehat{\lambda}}^{g^{-1}}\bigl[g^{-1}dg, S^{g^{-1}}\bigr],
$$
since $S$ is off-diagonal, and 
$$ -\tr D_{\widehat{\lambda}}^{g^{-1}}\bigl[g^{-1}dg, S^{g^{-1}}\bigr]= \tr \bigl[D_{\widehat{\lambda}},S\bigr]dgg^{-1}=\tr \bigl(g(\wh{\tau}+m)g^{-1}\bigr) dgg^{-1}=\tr (\wh{\tau}+m)g^{-1}dg.$$
Thus, to prove the result, we need to show that $d\tr (\wh{\tau}+m)g^{-1}dg=0$. We can write $\wh{\tau}=\tau\cdot 1-(k+1)\tau e_{k+1,k+1}$, and, consequently:
$$d\tr \wh{\tau}g^{-1}dg=d\tr \bigl(\tau\cdot 1-(k+1)\tau e_{k+1,k+1}\bigr)g^{-1}dg= $$
$$ \tau \tr d\bigl(g^{-1}dg\bigr)-(k+1)\tau d\tr e_{k+1,k+1}g^{-1}dg.$$
The first term is null and, hence, it is sufficient  to show that 
$\tr ng^{-1}dg=0$ for any $(k+1)\times(k+1)$-matrix $n$, the only nonzero entries of which are in the last column. We have, however, $\tr ng^{-1}dg=-\tr nd\bigl(g^{-1}\bigr)g$ and this vanishes, since the last row of $g^{-1}$ is constant (cf. \eqref{g^{ij}}). 
\end{proof}

Thus, restricting functions on $N_{k,\tau}\times\cx^2$ to $\tilde{N}_{k,\tau}^{\rm reg}$ induces a monomorphism of Poisson algebras:
\begin{equation} \cx\bigl[N_{k,\tau}\times\cx^2\bigr]\hookrightarrow 
\cx\left[T^\ast\fH_k^{\rm reg}\right]^{S_k} \otimes \cx\left[T^\ast\fH_{k+1}^{\rm reg}\right]^{S_{k+1}}.\label{Poisson}\end{equation}

Let $\Delta\in \cx\bigl[N_{k,\tau}\times \cx^2\bigr]$ be defined as
\begin{equation} \Delta(\wh{A},\wh{B})=\prod_{i\neq j}(\lambda_i-\lambda_j)\prod_{m\neq n}(\wh\lambda_i-\wh\lambda_j),\label{Delta}\end{equation}
where $\lambda_i$ are the eigenvalues of $A$ and $\wh\lambda_m$ are the eigenvalues of $\wh{A}$. $\Delta$ is also an $S_k\times S_{k+1}$-invariant function on $\fH_k\times \fH_{k+1}$ and, hence, on $T^\ast(\fH_k\times \fH_{k+1})$. Recall also that a localisation of a Poisson algebra by a multiplicative set is also a Poisson algebra.
\par
Since all denominators in the formulae \eqref{g_{ij}} and \eqref{g^{ij}} are factors of $\Delta$, the localisation of $\cx\bigl[N_{k,\tau}\times \cx^2\bigr]$ by the  multiplicative set generated by $\Delta$ is naturally identified with $\cx\bigl[\tilde{N}_{k,\tau}^{\rm reg}\bigr]$. Thus:
\begin{corollary} The localisation of $\cx\bigl[N_{k,\tau}\times \cx^2\bigr]$  by the  multiplicative set generated by $\Delta$ is isomorphic, as a Poisson algebra, to  $\left(\cx\bigl[T^\ast(\fH_k\times \fH_{k+1})\bigr]_\Delta\right)^{S_k\times S_{k+1}}$.\end{corollary}

 \begin{corollary} The rational function field of $N_{k,\tau}\times \cx^2$ is isomorphic, as a Poisson field, to the subfield of $(S_k\times S_{k+1})$-invariants in the rational function field of $T^\ast(\fH_k\times \fH_{k+1}) $.\label{field}\end{corollary}

\subsection{Complete integrability of instanton moduli spaces\label{integrable}}

As another application of Proposition \ref{symplectomorphism}, we observe that the instanton moduli spaces $N_{k,\tau}$ are algebraically completely integrable. Indeed, the dimension of $N_{k,\tau}$ is $4k$, and Proposition \ref{symplectomorphism} implies that the functions
\begin{equation} \Phi_i\bigl(\wh{A},\wh{B}\bigr)=\tr A^i,\quad \wh\Phi_j\bigl(\wh{A},\wh{B})=\tr \wh{A}^j,\label{Phi}\end{equation}
all Poisson commute, and that the only relation among $\Phi_1,\dots,\Phi_k, \wh\Phi_1,\dots,\wh\Phi_{k+1}$ is $\Phi_1=\wh\Phi_1$.

The corresponding abelian subgroup of symplectomorphisms of $N_{k,\tau}\times \cx^2$ is generated by 
\begin{equation}\bigl(\wh{A},\wh{B}\bigr)\mapsto  \bigl(\wh{A},\wh{B}+A^p\bigr),\quad \bigl(\wh{A},\wh{B}\bigr)\mapsto  \bigl(\wh{A},\wh{B}+\wh{A}^q\bigr),\quad p,q\in \oZ_{\geq 0},\label{group}\end{equation}
while on $N_{k,\tau}$, we have to project $\wh{B}+\wh{A}^q$ onto $\fS_{k+1}$, i.e. set the $(k+1,k+1)$-entry to $0$.

The group \eqref{group} is isomorphic to $\cx^{2k}$ and should be viewed as a $2$-step analogue of the Gelfand-Zeitlin group considered in \cite{KW}.

We have:
\begin{proposition} Suppose that both $A$ and $\wh{A}$ are regular semisimple and that $\wh{A}$ is $G$-regular. Then the action of the group \eqref{group} is transitive on the fibre $\Psi^{-1}\bigl(\Psi(\wh{A})\bigr)$, where $\Psi=\bigl(\Phi_1,\dots,\Phi_k, \wh\Phi_1,\dots,\wh\Phi_{k+1}\bigr)$.\label{trans}\end{proposition}
\begin{remark} Corollary \ref{rss} implies that the assumption is satisfied, if  $A$ and $\wh{A}$ are regular semisimple without a common eigenvalue.\end{remark}
\begin{proof} Let $P\subset \{1,\dots,k\}$ and write 
$$\fM_P=\{C\in \wh s; \enskip C_{ij}\neq 0\Rightarrow \text{$j=k+1$ $\&$ $i\in P$ or  $i=k+1$ $\&$ $j\not\in P$}\}.$$ 
Thus, $\fM_P\subset \fM$ and $\fM_\emptyset= \fM^-$. Let us also define $\wh\fG_P$ as the set of $C$, such that $C_{ij}=1$, if $j=k+1, i\not\in P$ or $i=k+1,j\in P$. 
According to the remarks after \eqref{canonical0}, the assumption implies that $\wh A$ can be $G$-conjugated to matrix, which lies in $\wh\fG_P$ for some $P$ and such that $A$ is diagonal. We can now repeat the proof of Proposition \ref{decompose2}, and show, that any $\wh B$, satisfying $\bigl[ \wh A,\wh B\bigr]\in \wh\tau+\fM$, can be decomposed as $B_1+B_2$, with $[A,B_1]=0$ and $[\wh A,B_2]\in \wh\tau +\fM_P$, with the entries of $[\wh A,B_2]$ determined by $\Psi(\wh A)$. Thus, $B_2$ can be written as in \eqref{gS}, and since $A,\wh A$ are regular semisimple, the action $B_1\mapsto B_1+A^p$, $B_2\mapsto B_2+\wh A^q$, $p,q\in \oZ_{\geq 0}$, is transitive on the set of $B_1,B_2$ corresponding to $\wh{A}$.
\end{proof}

\begin{remark} As remarked in the introduction, the complete integrability of $N_{k,\tau}$ follows also from general results about quiver varieties.\end{remark}

\subsection{The Hamiltonian\label{Ham}}

With the choice of coordinates given in section \ref{smor}, the natural Hamiltonian to consider is $H_\tau=\tr \wh B^2$. We compute $H_0$, which is also the quadratic (in the $\mu_i$-s and $\wh\mu_i$-s) part of any $H_\tau$: 
$$ \tr\wh B^2=\tr(B_1+B_2)^2=\tr B_1^2 +\tr B_2^2+2\tr B_1B_2=\tr D_\mu^2+\tr D_{\wh\mu}^2+ 2\tr D_\mu g^{-1}D_{\wh\mu}g.$$
Hence:
$$H_0=\sum_{i=0}^k\mu_i^2+\sum_{i=0}^{k+1}\wh\mu_i^2+2\sum_{i=1}^{k}\sum_{j=1}^{k+1}\frac{\prod_{n\neq j}(\lambda_i-\wh{\lambda}_n)\prod_{m\neq i}(\wh{\lambda}_j-{\lambda}_m)}{\prod_{m\neq i}({\lambda}_i-{\lambda}_m)\prod_{n\neq j}(\wh{\lambda}_j-\wh{\lambda}_n)}\mu_i\wh\mu_j.
$$

\bigskip

{\bf Part 2. Non-commutative symplectic geometry of the instanton quiver}

\medskip

Let  $Q$ be the quiver \eqref{Q2}. We label the two vertices as $1$ and $2$ (with the loop being at $1$). We label the arrows of $Q$ as $a,x,y$, with $a$ being the loop, $x$ going from $2$ to $1$ and $y$ from $1$ to $2$. The remaining arrows of $\ol Q$ are denoted by $a^\ast,x^\ast, y^\ast$, with $z^\ast$ being an arrow in the opposite direction to $z$. Let $V=\left(\cx^k,\cx^l\right)$ and denote, as usual, by $\sR(\ol Q,V)\simeq T^\ast \sR( Q,V)$ the symplectic vector space of all representations of $\ol{Q}$ in $V$. Thus, an element of $\sR(\ol Q,V)$ is $(A,B,X_1,X_2,Y_1,Y_2)$ with $A,B\in M_{k\times k}$, $X_1,X_2\in M_{k\times l}$, $Y_1,Y_2\in M_{l\times k}$. The group $PGL(V)=\bigl(GL(k,\cx)\times GL(l,\cx)\bigr)/\cx^\ast$ acts effectively on $\sR(\ol Q,V)$, inducing a Hamiltonian action on  $\sR(\ol Q,V)$, with moment map $\nu$ given by
\begin{equation} \bigl([A,B]+X_1Y_2-X_2Y_1,Y_1X_2-Y_2X_1\bigr)\in \gl(k,\cx)\oplus \gl(l,\cx).\label{C}\end{equation}
 For  an adjoint orbit $O$ of $PGL(V)$, we denote by $\sR_O(\ol Q,V)$ the symplectic quotient $\nu^{-1}(O)/PGL(V)$. 
\par
If we write 
\begin{equation} i_1=-X_1,\enskip i_2=X_2,\enskip j_1=Y_1,\enskip j_2=Y_2,\label{ij}\end{equation}
then the moment map becomes $\bigl([A,B]-i_1j_1-i_2j_2,j_1i_1+j_2i_2\bigr)$, and so the instanton moduli spaces $N_{k,\tau}$ are of the form $\sR_O(\ol Q,V)$, if we set $l=1$ and
$$O=\begin{pmatrix} \tau\cdot 1_{k\times k} & 0\\ 0 & -k\tau\end{pmatrix}.$$
We remark that replacing $X$-s and $Y$-s with $i$-s and $j$-s is equivalent to passing from the double of the quiver \eqref{Q2} to the double of the quiver \eqref{Q1}.
\par
In what follows, we shall discuss certain notions of the non-commutative symplectic geometry in the case of the quiver   \eqref{Q2}.

\section{The path algebra and the necklace algebra}

Let $\cx \ol Q$ be the path algebra of $\overline{Q}$, i.e. an algebra generated by the arrows in $\ol Q$ (including the trivial ones) with multiplication given by the concatenation of paths (paths are written from right to left). It is an algebra over the ring $R=\cx^2$, with the idempotents of $R$ corresponding to constant paths at the two vertices of $Q$. The necklace algebra $\sL Q$ is, as a vector space, $\cx \ol Q/\bigl[\cx\ol Q,\cx \ol Q\bigr]$, i.e. elements of $\cx\ol Q$ modulo cyclic permutations of arrows. To define the Lie algebra bracket, observe first that, for every $w\in \ol Q$, there is a $\cx$-linear map
\begin{equation} \frac{\partial}{\partial w}:\cx \ol Q/\bigl[\cx\ol Q,\cx \ol Q\bigr] \rightarrow \cx \ol Q,\end{equation}
given on arrows by $\frac{\partial w}{\partial w}=1$, $\frac{\partial u}{\partial w}=0$, if $w\neq u\in \ol Q$, and, in general, by:
\begin{equation}  \frac{\partial}{\partial w}u_1\dots u_n=\sum_{i=1}^n \frac{\partial u_i}{\partial w}u_{i+1}\dots u_nu_1\dots u_{i-1}.\label{d/dw}\end{equation}
The Lie algebra bracket on $\sL Q=\cx \ol Q/\bigl[\cx\ol Q,\cx \ol Q\bigr]$
is then defined by \cite{LBB,Gin}:
\begin{equation} \{f,g\}_Q=\sum_{z\in Q}\left(\frac{\partial f}{\partial z}\frac{\partial g}{\partial z^\ast}-\frac{\partial f}{\partial z^\ast}\frac{\partial g}{\partial z}\right) \mod \bigl[\cx\ol Q,\cx \ol Q\bigr],\label{Q-bracket} \end{equation}
where the multiplication is in $\cx \ol Q$.
Ginzburg \cite[Proposition 3.4]{Gin} shows that $\sL Q$ is a central extension of the Lie algebra of symplectic derivations of $\cx \ol Q$.
\par
There exists a Lie algebra homomorphism from $\sL Q$ to the Poisson algebra of algebraic functions on any quiver variety for $\ol Q$.
 Namely, one considers $\cx^{k+l}=\cx^k\oplus \cx^l$ as an $R$-module, which makes $\gl(k+l,\cx)$ an $R$-bimodule. It is equipped with the map $\tr_R:\gl(k+l,\cx)\rightarrow R=\cx^2$, defined as taking separately the traces of the upper-left $k\times k$- and the lower-right $l\times l$-block. The canonical $R$-algebra homomorphism
\begin{equation} E:\cx \ol Q\longrightarrow \cx\left[\sR(\ol Q,V)\right]\otimes_R \gl(k+l,\cx)\label{evaluation}\end{equation}
evaluates each non-commutative polynomial in $\cx \ol Q$ on  the matrices corresponding to a point in $\sR(\ol Q,V)$. Taking now the trace $\tr_R$ on the second factor gives a map
\begin{equation}\wh{\tr}: \cx \ol Q\rightarrow \cx\left[\sR(\ol Q,V)\right]^{GL(V)},\label{hattr}\end{equation}
which clearly vanishes on $\bigl[\cx\ol Q,\cx \ol Q\bigr]$ and it induces a Lie algebra homomorphism from $\sL Q$ to $\cx\left[\sR(\ol Q,V)\right]^{GL(V)}$. According to a result of Le Bruyn and Procesi \cite{LBP}, $\cx\left[\sR(\ol Q,V)\right]^{GL(V)}$ is generated, as a $\cx$-algebra, by the image of $\wh\tr$.
A consequence of this fact, proved by Ginzburg \cite{Gin} and Bocklandt and Le Bruyn \cite{LBB}, is that a smooth quiver variety $\sR_O(\ol Q,V)$ is a  coadjoint orbit in the necklace Lie algebra of $Q$ (i.e. the Lie algebra action of $\sL Q$ on $\sR_O(\ol Q,V)$ is locally transitive).

\subsection{The necklace algebra as a non-commutative $\cx^2\times\gl(2,\cx)$}

We keep the notation of the previous section, in particular, the labelling of vertices and arrows in $\ol Q$. In addition, let $\pi_1,\pi_2$ be the two idempotents in $R=\cx^2$, corresponding to trivial paths at the vertices $1,2$ of $Q$. In the path algebra $\cx Q$ we have: $\pi_1 a=a\pi_1=a$, $\pi_1 x=x=x\pi_2$, $\pi_2 y=y=y\pi_1$, and the remaining products involving  $\pi_i$-s vanish. 
\par
 Let $\sA_1=\pi_1 \cx \ol Q \pi_1$ be the subalgebra of $\cx \ol Q$ generated by paths beginning and ending at the vertex $1$. Obviously, $\sA_1$ is a free $\cx$-algebra on $a,a^\ast, xx^\ast,xy, y^\ast x^\ast, yy^\ast$. We observe that
$$ \cx \ol Q/\bigl(\cx \pi_2+[\cx \ol Q,\cx \ol Q]\bigr)\simeq \sA_1/[\sA_1,\sA_1] \quad \text{as vector spaces}.$$
  Before computing the induced bracket on $\sA_1/[\sA_1,\sA_1]$, let us introduce a matrix
\begin{equation} E=\begin{pmatrix} e_{11} & e_{12}\\ e_{21}& e_{22}\end{pmatrix} =
\begin{pmatrix} -xx^\ast & -xy\\ y^\ast x^\ast& y^\ast y\end{pmatrix}= \begin{pmatrix} -x\\ y^\ast\end{pmatrix} \begin{pmatrix} x^\ast & 
y\end{pmatrix}. \label{e-s}\end{equation}
Thus, $\sA_1\simeq \cx\langle a,a^\ast,e_{11},e_{12},e_{21},e_{22}\rangle$. We compute the Lie bracket on the generators of $\sA_1$:
\begin{equation} \{a,a^\ast\}=1,\quad \{a,e_{ij}\}=\{a^\ast,e_{ij}\}=0,\quad \{e_{ij},e_{kl}\}=\delta_{jk}e_{il}-\delta_{il}e_{kj},\label{e-e} \end{equation}
which is just the linear part of the Poisson structure on $\cx^2\times \gl(2,\cx)$. The formula \eqref{Q-bracket} implies now:
\begin{proposition} The abelianisation map $$\cx\langle a,a^\ast,e_{11},e_{12},e_{21},e_{22}\rangle\rightarrow \cx[a,a^\ast,e_{11},e_{12},e_{21},e_{22}]$$ induces a surjective Lie algebra homomorphism $\sL Q/\cx \pi_2\rightarrow \cx\bigl[\cx^2\times \gl(2,\cx)\bigr]$, where the Lie bracket on the latter algebra is the standard Poisson structure of $\cx^2\times \gl(2,\cx)$.\hfill $\Box$\end{proposition}

We also have
\begin{proposition} The algebra $\cx Q/[\cx Q,\cx Q]$ is a maximal commutative subalgebra of $\sL Q$.\hfill $\Box$\label{lie-centr}\end{proposition}

\section{The group of symplectic automorphisms of $\cx\ol Q$}
We consider the group $\Aut_R\cx \ol Q$ of $R$-automorphisms of $\cx \ol Q$. Just as for the Calogero-Moser quiver, $\Aut_R\cx \ol Q$ is an algebraic ind-group. One considers a filtration of  $\cx\ol Q$ by degree (i.e. the length of paths) and one views $\Aut_R\cx \ol Q$ as a closed variety in $\bigl(\cx\ol Q\bigr)^{\oplus 12}$ via the map, which associates to $\phi\in \Aut_R\cx \ol Q$ the values of $\phi$ and of $\phi^{-1}$ on the generators $a,a^\ast,x,x^\ast,y,y^\ast$ (cf. \cite{Sh}). 
\par
For any element $p\in \cx \ol Q$, we consider its stabiliser
$$ \Aut_R\bigl(\cx \ol Q;p\bigr)=\left\{\phi\in \Aut_R\cx \ol Q;\enskip \phi(p)=p\right\}.$$
 It is a closed subgroup of $\Aut_R\cx \ol Q$, and an algebraic ind-group with the respect to the induced filtration. Of particular importance is the group $ \Aut_R\bigl(\cx \ol Q;c\bigr)$, where 
\begin{equation} c=[a,a^\ast]+ [x,x^\ast]+[y,y^\ast].\label{c1}\end{equation}
It is the group of automorphisms preserving the non-commutative symplectic form of $\cx\ol Q$, in the sense of \cite{Gin}. We shall refer to $ \Aut_R\bigl(\cx \ol Q;c\bigr)$ as the {\em group of symplectic automorphisms}.
Its Lie algebra is a subalgebra 
of the algebra of symplectic derivations of $\cx \ol Q$, i.e.  a subalgebra of $\sL Q/R$. We shall see that this  is a proper inclusion, just  as in the case of the Calogero-Moser quiver.
\par
The group $\sG=\Aut_R\cx \ol Q$ acts on  $\sR(\ol Q,V)$: for $g\in \sG$ and any arrow $v$ in $\ol Q$, evaluate the non-commutative polynomial $g(v)$ on $(A,B,X_1,Y_1,X_2,Y_2)$.
It follows that  $ \Aut_R\bigl(\cx \ol Q;c\bigr)$ preserves the moment map \eqref{C}  for the action of $PGL(V)$.
It is also clear that the action of $ \Aut_R\cx \ol Q$ commutes with the  action of $PGL(V)$, and, hence, $ \Aut_R\bigl(\cx \ol Q;c\bigr)$ acts on each quiver variety  $\sR_O(\ol Q,V)$, in particular on the instanton moduli space $N_{k,\tau}$, via symplectomorphisms. It is easy to see that this action is algebraic (see \cite{Ku} for more on algebraic ind-groups and their algebraic actions).

The element  $c=[a,a^\ast]+ [x,x^\ast]+[y,y^\ast]$ is the sum of 
\begin{equation} c_1=[a,a^\ast]+xx^\ast-y^\ast y\quad \text{and}\quad c_2=yy^\ast -x^\ast x,\label{c12}\end{equation} 
and any $R$-automorphism, which preserves $c$ must preserve both $c_1$ and $c_2$. Thus
$$\Aut_R\bigl(\cx \ol Q;c\bigr)=\Aut_R\bigl(\cx \ol Q;c_1\bigr)\cap \Aut_R\bigl(\cx \ol Q;c_2\bigr).$$

\subsection{$\Aut_R\cx \ol Q$ as a semi-direct product}
Any $R$-automorphism of $\cx \ol Q$ preserves the ideal $I$ generated by $x,x^\ast,y,y^\ast$, and hence induces an automorphism of
$$ \cx \ol Q/I\simeq \cx\langle a,a^\ast\rangle.$$
The induced map
$$ \Aut_R\cx \ol Q\longrightarrow  \Aut\cx\langle a,a^\ast\rangle$$
is a morphism of algebraic ind-groups. On the other hand, any automorphism of  $\cx\langle a,a^\ast\rangle$ extends to an automorphism of $\cx \ol Q$, acting as the identity on $x,x^\ast,y,y^\ast$. Thus, $\Aut_R\cx \ol Q$ is a split extension (i.e. a semi-direct product):
\begin{equation} 1\rightarrow \sK\rightarrow \Aut_R\cx \ol Q\rightleftarrows \Aut\cx\langle a,a^\ast\rangle\rightarrow 1\label{split}\end{equation}
of algebraic ind-groups. We can restrict this exact sequence to any subgroup $\Aut_R\bigl(\cx \ol Q;p\bigr)$; in particular we have:
\begin{equation} 1\rightarrow \sK_c\rightarrow \Aut_R\bigl(\cx \ol Q;c\bigr)\rightleftarrows \Aut\bigl(\cx\langle a,a^\ast\rangle ; [a,a^\ast]\bigr)\rightarrow 1.\label{split2}\end{equation}

\subsection{ The Lie algebra vs. symplectic derivations\label{GneqL}}

A consequence of  \eqref{split2} is that the Lie algebra of $ \Aut_R\bigl(\cx \ol Q;c\bigr)$ cannot be the full Lie algebra of symplectic derivations of $\cx \ol Q$ (i.e.  $\sL Q/\R$). This is known for the Calogero-Moser quiver, where it follows from the Czerniakiewicz-Makar-Limanov Theorem on the structure of $\Aut\cx\langle u,v\rangle$. This theorem (see, e.g., \cite[Theorem 6.10.5]{Cohn}) implies that any automorphism of $\cx\langle u,v\rangle$ takes $u$ and $v$ to palindromic words (i.e. unchanged, when written backwards). The same must be true for any derivation in $\Lie\bigl( \Aut\cx\langle u,v\rangle\bigr)$.
\par
Now, the exact sequence  \eqref{split2} induces an analogous exact sequence on the Lie algebras. Let $D$ be a symplectic (i.e. killing $[a,a^\ast]$) derivation of $\cx\langle a,a^\ast\rangle$, which is not in $\Lie\bigl( \Aut\cx\langle a,a^\ast\rangle\bigr)$. Extend $D$ to a symplectic derivation $\tilde{D}$ of $\cx \ol Q$ by setting $D(u)=0$ for $u=x,y,x^\ast,y^\ast$. The sequence  \eqref{split2} implies that $\tilde{D}$ cannot belong to the Lie algebra of $ \Aut_R\bigl(\cx \ol Q;c\bigr)$ (in fact, using \eqref{split} instead, not even to $\Lie\bigl(\Aut_R\cx \ol Q\bigr)$).

\subsection{More on  $\Aut_R\bigl(\cx \ol Q;c_2\bigr)$ } Recall the subalgebra $\sA_1\subset \cx \ol Q$, which is isomorphic to the free $\cx$-algebra on $6$ letters  $a,a^\ast,e_{11},e_{12},e_{21},e_{22}$, where the $e_{ij}$ are defined by \eqref{e-s}. We wish, analogous to what we did for the necklace algebra, to describe symplectic automorphisms of $\cx \ol Q$ as automorphisms of $\sA_1$. We can actually do this for the bigger group  $\Aut_R\bigl(\cx \ol Q;c_2\bigr)$.

It is clear that any $R$-automorphism of $\cx \ol Q$ preserves $\sA_1$ and so we have a homomorphism
\begin{equation} \Aut_R\bigl(\cx \ol Q\bigr)\rightarrow \Aut_\cx(\sA_1).\label{Q-A}\end{equation}
For a $\psi \in  \Aut_R\bigl(\cx \ol Q\bigr)$ and an $u\in \cx \ol Q$,  we shall write $u^\psi$ for $\psi(u)$, and $E^\psi$ for $\bigl[e_{ij}^{\psi}\bigr]$.
We have:
\begin{proposition} The kernel of the restriction of the homomorphism \eqref{Q-A} to $\Aut_R\bigl(\cx \ol Q;c_2\bigr)$   consists of automorphisms $x\mapsto \lambda x, y^\ast\mapsto \lambda y^\ast,  x^\ast\mapsto \lambda^{-1} x^\ast, y\mapsto \lambda^{-1} y$, $\lambda\in \cx^\ast$. The image
 consists of automorphisms $\phi$ of  $\sA_1$, such that $E^\phi= M_\phi E M_\phi^{-1}$, for some $M_\phi\in GL_2(\sA_1)$. \label{Gc2}\end{proposition}
\begin{remark} The map $\Aut_R\bigl(\cx \ol Q;c_2\bigr)\rightarrow GL_2(\sA_1)$ is not a homomorphism.\end{remark}
\begin{proof} Let $\psi\in \Aut_R\cx \ol Q$ preserve $c_2$. Since $\psi$ is an $R$-automorphism, we can write
$$ \begin{pmatrix} -x^\psi\\ (y^\psi)^\ast\end{pmatrix}=M_\psi \begin{pmatrix} -x \\ y^\ast\end{pmatrix},\quad  \begin{pmatrix} (x^\psi)^\ast& 
y^\psi\end{pmatrix}=
 \begin{pmatrix} x^\ast &
y\end{pmatrix}N_\psi,$$
where $M_\psi,N_\psi\in GL_2(\sA_1)$. The condition of preserving $c_2$ can now be rewritten as 
$$         \begin{pmatrix} x^\ast &
y\end{pmatrix}N_\psi           M_\psi\begin{pmatrix} -x \\ y^\ast\end{pmatrix}=c_2.$$
We claim that $N_\psi M_\psi=1$. Let us write $C=[c_{ij}]=N_\psi M_\psi-1$. The last equation can be rewritten as
$$ \begin{pmatrix} x^\ast &
y\end{pmatrix}C \begin{pmatrix} -x \\ y^\ast\end{pmatrix}=0,$$
i.e.:
$$-x^\ast c_{11} x +x^\ast c_{12}y^\ast-yc_{21}x+yc_{22}y^\ast=0.$$
Multiplying this by $x$ on the left and by $x^\ast$ on the right yields 
$$ e_{11}(c_{11}e_{11}+c_{12}e_{12})+e_{12}(c_{21}e_{11}+c_{22}e_{21})=0.$$ 
Therefore (as $\sA_1$ is a free algebra) $c_{11}e_{11}+c_{12}e_{12}=0$ and $c_{21}e_{11}+c_{22}e_{21}=0$, and, hence, $C=0$. Thus, $N_\psi=M_\psi^{-1}$ and equation \eqref{e-s} implies that the induced automorphism $\phi$ of $\sA_1$ transforms $E$ as in the statement, with $M_{\phi}=M_\psi$.
\par
 The kernel of the homomorphism $\psi\mapsto\phi$ consists of those $\psi$, for which $M_\psi$ commutes with $E$, and the proof will be complete once we prove that the only matrices in $ GL_2(\sA_1)$, which commute with $E$, are of the form  $\begin{pmatrix} \lambda & 0\\0 &\lambda\end{pmatrix}$ for some $\lambda\in \cx^\ast$.  
For an element $t$ of $\sA_1$, denote by $\supp t$ the set of all distinct nonzero monomials making up $t$. 
\begin{lemma} Let $M=[m_{ij}]$ is a $2\times 2$-matrix with entries in $\sA_1$ which commutes with $E$. Suppose that $\supp m_{11}$  contains a word $w$ of degree $n>0$. Then $\supp m_{11}$  contains $\alpha (e_{11})^n$ ( $\alpha \neq 0$).\end{lemma}
\begin{proof}  Consider the $(11)$-component of the equation $ME^n=E^nM$. The word $w(e_{11})^n$ on the left-hand side must cancel a word of the form $e_{1i_1}\dots e_{i_{n-1}i_n}u$ on the right-hand side, where $u\in \supp m_{i_n1}$. Comparing the degrees, we obtain that $u=\alpha(e_{11})^n$. We need to show that $i_n=1$. Consider  the $(21)$-component of the equation $ME=EM$, i.e. $m_{21}e_{11}+m_{22}e_{21}=e_{21}m_{11}+e_{22}m_{21}$. From this,  it is clear that $\supp m_{21}$ cannot contain $(e_{11})^n$.\end{proof}
Let $[M,E]=0$ and let $\alpha (e_{11})^n\in \supp m_{11}$. The matrix $M^\prime=M-\alpha E^n$ also commutes with $E$, but the support of its $(11)$-entry $m_{11}^\prime$ does not contain any $\beta (e_{11})^n$.  The lemma implies that $\supp m_{11}^\prime$ does not contain any words of degree $n$. Doing this for every degree, we can write $M=P+Q$, where $P=\sum_{j=1}^k \alpha_j E^j$,  $QE^i=E^iQ$ for all $i$, and $Q_{11}=0$.  We claim that $Q=0$. Indeed, the $(11)$- and $(12)$-components of the equation $QE=EQ$ reduce to
$$ Q_{12}e_{21}=e_{12}Q_{21},\quad Q_{12}e_{22}=e_{12}Q_{22}+e_{11}Q_{12}.$$
The first equation implies that every word in $\supp Q_{12}$ begins with $e_{12}$, and, hence, the second equation implies that $e_{11}Q_{12}$ is divisible on the left by  $e_{12}$, which is possible only if $Q_{12}=0$. The two equations imply now  that  $Q_{21}=0$ and $Q_{22}=0$.
\par
Thus, we conclude that, if  $ME=EM$, then $M=\sum_{j=1}^k \alpha_j E^j$ for some scalars $\alpha_{j}$. Suppose now that $M$ is in addition
invertible. Then $M^{-1}$ also commutes with $E$, and, hence, $M^{-1}=\sum_{j=1}^l \beta_j E^j$ for some $\beta_{j}$. Since $E^s$, $s=1,2,\dots$,  are
linearly independent, $MM^{-1}=1$ implies that $k=l=0$. Therefore $M=\alpha_0\cdot 1$, for some $\alpha_0\neq 0$.
\end{proof}

\section{Triangular and tame automorphisms} 

\begin{definition} An automorphism $\phi\in \Aut_R\cx \ol Q$ is said to be  {\em triangular} (resp. {\em strictly triangular}), if $\phi\bigl(\cx Q\bigr)=\cx Q$ (resp. $\phi$ is identity on $\cx Q$).\end{definition}
We shall describe now the group of strictly triangular automorphisms preserving $c$. Let $F_2=\cx\langle a,b\rangle $ be the free algebra on two letters $a$ and $b=xy$. Write
\begin{equation} L_2=F_2/\bigl(\cx +[F_2,F_2]\bigr).\end{equation}
We view $L_2$ as an abelian group, with respect to addition.
Recall the formula \eqref{d/dw} giving us maps
$$ \frac{\partial}{\partial a},\frac{\partial}{\partial b}:L_2\rightarrow F_2.$$
We introduce a map
$$ \Lambda:L_2\rightarrow \Aut_R\cx\ol Q,$$
defined on generators by:
\begin{eqnarray*} \Lambda(f)(a,x,y) & = & (a,x, y)\\
  \Lambda(f)(a^\ast)&=&a^\ast+\frac{\partial f}{\partial a}\\
\Lambda(f)(x^\ast)&=&x^\ast+ y\frac{\partial f}{\partial b} \\
\Lambda(f)(y^\ast)&=&y^\ast+\frac{\partial f}{\partial b}x, \end{eqnarray*}
for every non-commutative polynomial $f(a,b)$, $b=xy$.

\begin{proposition} Every $\Lambda(f)$ preserves the commutator \eqref{c1}, and the image of $\Lambda$ coincides with the stabiliser of the triple $(a,x,y)$ in $\Aut_R\bigl(\cx \ol Q;c\bigr)$. \label{stabiliser}\end{proposition}
\begin{remark} According to Proposition \ref{Gc2}, $\Lambda(f)$ induces an automorphism of $A_1\simeq 
\cx\langle a,a^\ast,e_{11},e_{12},e_{21},e_{22}\rangle$, such that the action on the $e_{ij}$ is given by conjugating the matrix $E=[e_{ij}]$ by an $M\in GL_2(\sA_1)$. This matrix for $\Lambda(f)$ is $$M=\begin{pmatrix} 1 & 0\\ -\frac{\partial f}{\partial b} & 1\end{pmatrix}.$$
\end{remark}
\begin{proof} Let $\phi\in \sG_c$ with $\phi (a,x,y)=(a,x,y)$. Put $\phi(a^\ast,x^\ast,y^\ast)=(a^\ast+h,x^\ast+s,y^\ast+t)$, where the endpoints of $h,s,t$ are the same as for $a^\ast,x^\ast,y^\ast$. The condition of preserving $c$ can be rewritten as:
\begin{equation} [a,h]+[x,s]+[y,t]=0,\label{hst}  \end{equation}
which is decomposed, using the idempotents, as:
 \begin{equation}[a,h]+xs-ty=0,\quad yt-sx=0.\label{hst2}  \end{equation}
As in the previous section, let $\sA_1\simeq \cx\langle a,a^\ast,e_{11},e_{12},e_{21},e_{22}\rangle$ be the subalgebra of $\cx \ol Q$ generated by all paths beginning and ending at the vertex $1$ (the $e_{ij}$-s are given by \eqref{e-s}).
The second equation in \eqref{hst2} implies that $s,t$ can be written as $s=yu$ and $t=vx$, with $u,v\in A_1$, and, hence $u=v$.  The first equation becomes
$[a,h]+[xy,u]=0$. Thus, the solutions to \eqref{hst} are in $1-1$ correspondence with elements $h,u\in A_1$, which satisfy
\begin{equation} [a,h]-[e_{12},u]=0.\label{free}\end{equation}
 Again, denote by $\supp t$, $t\in \sA_1$,  the set of all distinct nonzero monomials making up $t$. Let us also write $b=-e_{12}=xy$, so that  \eqref{free} becomes $[a,h]+[b,u]=0$. It is easy to see that this equation is equivalent to the following conditions, for any monomial $m$:
\begin{equation*} f\in\{a,a^\ast,e_{ij}; i,j=1,2\}\enskip \&\enskip \{fm,mf\}\cap \bigl(\supp h\cup \supp u\bigr)\neq \emptyset\Rightarrow \text{$f=a$ or $f=e_{12}$},\end{equation*}
\begin{equation} 
ma\in \supp h \Leftrightarrow am\in \supp h,\quad mb\in \supp u \Leftrightarrow bm\in \supp u,\label{ad}\end{equation}
\begin{equation} 
mb\in \supp h \Leftrightarrow am\in \supp u,\quad bm\in \supp h \Leftrightarrow ma\in \supp u.\label{da}\end{equation}
In particular, $h,u\in \cx\langle a,b\rangle$.
Let $w=\lambda a^{i_1}b^{j_1}\dots a^{i_n}b^{j_n}$, $\lambda\neq 0$, be any monomial in $S=\supp h\cup \supp u$, and consider the minimal subset $S_w$ of $S$, containing $w$, and invariant under the operations given in \eqref{ad} and \eqref{da}. Let $h_w$ be the sum of all monomials in $S_w\cap \supp h$ and $u_w$ the sum of all monomials in $S_w\cap \supp u$.
Then $h_w,u_w$ is a solution of \eqref{free}, and to finish the proof observe that the invariance under \eqref{ad} and \eqref{da} is equivalent to $(h_w,u_w)=\Lambda(aw)$, if $w\in \supp h$, and  $(h_w,u_w)=\Lambda(bw)$, if $w\in \supp u$.
\end{proof}

\begin{remark} The above proof implies the following: two elements $p,q\in \cx\langle a,b\rangle$ satisfy $[a,p]+[b,q]=0$ if and only if there exists an $f\in F_2/[F_2,F_2]$, such that $p=\frac{\partial f}{\partial a}$ and $q=\frac{\partial f}{\partial b}$ (this also follows from the fact that $[a,p]+[b,q]=0$ is equivalent to $a\mapsto -q, b\mapsto p$ being a symplectic derivation of $\cx\langle a,b\rangle$). Thus, the group of strictly triangular automorphisms of $\cx \ol Q_8$, where $Q_8$ is the quiver having one vertex and two loops, is isomorphic to the group of strictly triangular automorphisms of $\cx \ol Q$.
\label{F_2}\end{remark}

\begin{remark} It is easy to see that the symplectomorphisms \eqref{group} of the instanton moduli spaces $N_{k,\tau}$ are induced by strictly triangular automorphisms in $\Aut_R\bigl(\cx \ol Q;c\bigr)$.\end{remark}

There is a second obvious subgroup of $\Aut_R\bigl(\cx \ol Q;c\bigr)$: the group $\text{\rm Aff}_c$ consisting of the affine transformations of $\Span (a,a^\ast,x,y,x^\ast,y^\ast)$ preserving $c$. Since these must be $R$-homomorphisms, it follows that 
\begin{equation} \text{\rm Aff}_c=ASL(2,\cx)\times GL(2,\cx),\label{aff}\end{equation} 
where the group $ASL(2,\cx)$ of unimodular affine transformations of $\cx^2$ acts in the usual way on $\Span (a,a^\ast)$, while $T\in GL(2,\cx)$ acts only on the $x,y,x^\ast,y^\ast$, via:
\begin{equation} \begin{pmatrix} -x\\ y^\ast\end{pmatrix}\mapsto T \begin{pmatrix} -x \\ y^\ast\end{pmatrix},\quad  \begin{pmatrix} x^\ast & 
y\end{pmatrix}\mapsto 
 \begin{pmatrix} x^\ast &
y\end{pmatrix}T^{-1}.\label{affine}\end{equation}

Let $\TAut_R\bigl(\cx \ol Q;c\bigr)$ be the subgroup of $\Aut_R\bigl(\cx \ol Q;c\bigr)$ generated by strictly triangular automorphisms  and by $\text{\rm Aff}_c$. It is easy to see that any triangular automorphism is in  $\TAut_R\bigl(\cx \ol Q;c\bigr)$. We call  $\TAut_R\bigl(\cx \ol Q;c\bigr)$ {\em the group of tame symplectomorphisms}. We do not know whether  every symplectic automorphism of  $\cx \ol Q$ is tame.

\section{Transitivity of $\Aut_R\bigl(\cx \ol Q;c\bigr)$ on $N_{k,\tau}$}

We shall now prove
\begin{theorem} The group $\TAut_R\bigl(\cx \ol Q;c\bigr)$ acts transitively on $N_{k,\tau}$, if $\tau\neq 0$.\label{transitive}\end{theorem}
\begin{remark} The corresponding result for the Calogero-Moser quiver has been proved by Berest and Wilson \cite{BW}.\end{remark}
\begin{remark} In the case of the Calogero-Moser quiver $Q_{\rm CM}$, the action of $\Aut\bigl(\cx\langle a,a^\ast\rangle ; [a,a^\ast]\bigr)$ has an open orbit on every quiver variety $\sR_O(\ol Q_{\rm CM},V)$; see \cite{A}. This is {\em not} the case for $\Aut_R\bigl(\cx \ol Q;c\bigr)$.  Indeed,  Proposition \ref{Gc2} implies that the action on $\sR_O(\ol Q,V)$ preserves the conjugacy class of the $2k\times 2k$-matrix \eqref{ET} below, which has rank $l$, if $V=(\cx^k,\cx^l)$.\end{remark}

Before proving Theorem \ref{transitive}, let us show how Theorem 3 of the introduction follows from it. Let $\text{\rm SAff}_c= ASL(2,\cx)\times SL(2,\cx)$, and let $\text{\rm STAut}_R\bigl(\cx \ol Q;c\bigr)$ be generated by strictly triangular automorphisms  and by $\text{\rm SAff}_c$. We observe that $\text{\rm STAut}_R\bigl(\cx \ol Q;c\bigr)$ is also generated by strictly triangular automorphisms and by ``opposite" strictly triangular automorphisms, i.e. those $\phi$, which satisfy $\phi(a^\ast,x^\ast,y^\ast)=(a^\ast,x^\ast,y^\ast)$. Thus, the Lie algebra $\fL$ of $\text{\rm STAut}_R\bigl(\cx \ol Q;c\bigr)$ is generated by the two commutative subalgebras $\cx Q/[\cx Q,\cx Q]$ and $\cx Q^{\rm op}/[\cx Q^{\rm op},\cx Q^{\rm op}]$. 
We also observe that the action of the centre of $GL(2,\cx)$ is trivial on $N_{k,\tau}$, and, hence, the group $\TAut_R\bigl(\cx \ol Q;c\bigr)$ can be replaced by  $\text{\rm STAut}_R\bigl(\cx \ol Q;c\bigr)$ in the statement of  Theorem \ref{transitive}. Therefore, the symplectic manifold $N_{k,\tau}$ is a coadjoint orbit of $\text{\rm STAut}_R\bigl(\cx \ol Q;c\bigr)$, and, consequently, the algebra of polynomial functions on $N_{k,\tau}$ is a quotient of the Poisson algebra of polynomial functions on $\fL^\ast$, i.e. a quotient of the symmetric algebra $S\fL$. Therefore $\cx[N_{k,\tau}]$ is generated, as a Poisson algebra, by $S\bigl(\cx Q/[\cx Q,\cx Q]\bigr)$ and  $S\bigl(\cx Q^{\rm op}/[\cx Q^{\rm op},\cx Q^{\rm op}]\bigr)$, which is 
exactly the statement of Theorem 3.

\medskip 

The remainder of the section will be devoted to a proof of Theorem \ref{transitive}. 

We are going to use the representation \eqref{ij} of $\ol Q$, i.e.:
\begin{equation} a\mapsto A,\enskip a^\ast\mapsto B,\enskip x\mapsto - i_1, \enskip y\mapsto j_2,\enskip x^\ast\mapsto j_1, \enskip y^\ast\mapsto i_2.\label{WWW}\end{equation}
We also write $i=\begin{pmatrix}i_1 &  i_2\end{pmatrix}\in \Hom\bigl(\cx^{2},\cx^k\bigr)$ and $j=\begin{pmatrix}j_1\\ j_2\end{pmatrix}\in \Hom\bigl(\cx^{k},\cx^{2}\bigr)$. We observe that the matrix $E$ given by \eqref{e-s} maps to 
\begin{equation} \begin{pmatrix}i_1j_1 &  i_1j_2\\ i_2j_1 & i_2j_2\end{pmatrix}.\label{ET}\end{equation}

Inside $N_{k,\tau}$, we have the subset 
\begin{equation} M_{k,\tau}=\left\{[A,B,i,j]\in N_{k,\tau};\enskip i_2=0,j_2=0\right\}.\label{MM}\end{equation}
$M_{k,\tau}$ is isomorphic to the Calogero-Moser space \cite{W}. In particular, for $\tau\neq 0$, Berest and Wilson \cite{BW} have shown that the group $\Aut\bigl(\cx\langle a,a^\ast\rangle ; [a,a^\ast]\bigr)$ acts transitively on $M_{k,\tau}$.  
 It follows that to prove the theorem, it is sufficient to move any point of $N_{k,\tau}$ into $M_{k,\tau}$.
\par
We begin with
\begin{lemma} Let $\tau\neq 0$ and $m=[A,B,i,j]\in N_{k,\tau}$ be such that $A$ is regular semisimple. Then there exists a $g\in \TAut_R\bigl(\cx \ol Q;c\bigr)$ such that $gm\in M_{k,\tau}$.\label{Omax}\end{lemma} 
\begin{proof} Let $p(a)$ be any polynomial in the variable $a$. We consider the strictly triangular automorphism $T_p=\Lambda(f)\in \Aut_R\cx \ol Q$ corresponding to $f(a,b)=-p(a)b$. In particular  $\frac{\partial f}{\partial b}=-p(a)$. It follows from the previous section that the action of $T_p$ on  $i_1,i_2,j_1,j_2$ is:
\begin{equation} \begin{pmatrix}i_1\\ i_2\end{pmatrix}\mapsto    \begin{pmatrix} 1 & 0\\ p(A) & 1\end{pmatrix} \begin{pmatrix}i_1\\ i_2\end{pmatrix},\quad \begin{pmatrix}j_1 &  j_2\end{pmatrix} \mapsto   \begin{pmatrix}j_1 &  j_2\end{pmatrix}\begin{pmatrix} 1 & 0\\-p(A) & 1\end{pmatrix}. \label{action}\end{equation}
We can assume that $A$ is diagonal, with distinct eigenvalues. Since $[A,B]-ij=\tau\cdot 1$ and $\tau\neq 0$, we can use the $GL(2,\cx)$-action on $i,j$ in order to guarantee that all entries of $i_1$ are nonzero. We can then find a polynomial $p(A)$ such that $i_2+p(A)i_1=0$. Thus, we can assume that $i_2=0$. Therefore $[A,B]=\tau\cdot 1+i_1j_1$, and, since $A$ is diagonal, all entries of $j_1$ are nonzero. We now use the $GL(2,\cx)$-action to send $i$ and $j$ to $(i_1^\prime,i_2^\prime)=(0,i_1)$ and $(j_1^\prime,j_2^\prime)=(j_2,j_1)$. The action \eqref{action} preserves the condition $i_1^\prime=0$, and, since all entries of $j_2^\prime$ are non-zero, we can find a polynomial $p(A)$, which sends $j_1^\prime$ to $0$.
\end{proof}

We now decompose $N_{k,\tau}$ as the union $N_{k,\tau}=\bigcup N_O$, where $O$ runs over all adjoint orbits of rank $\leq 2$ matrices and
\begin{equation} N_O=\bigl\{ (A,B,i,j); \enskip [A,B]-\tau\cdot 1=ij, \enskip ij\in O\bigr\}/GL(k,\cx).
\label{OO}\end{equation}
Thus, 
$O$ is an orbit of
\begin{equation} \diag(\alpha_1,\alpha_2,0,\dots,0),\label{albe1}\end{equation}
or of
 \begin{equation} E_{12}+\diag(\alpha,\alpha,0,\dots,0).\label{albe2}\end{equation}
In addition, $N_O$ is nonempty if and only if $\alpha_1+\alpha_2=-k\tau$ for \eqref{albe1}, and $2\alpha=-k\tau$ for \eqref{albe2}. Thus, the $\alpha$-s are determined by $ \tr(ij)^2=\tr(ji)^2$, and in particular, $\tr(ji)^2=k^2\tau^2$ implies that $O$ is the orbit of $\diag(-k\tau,0,0,\dots,0)$, i.e. $[A,B]-\tau \cdot 1$ has rank $1$. Denote this particular $N_O$ by $N_1$.  We observe that any point of $N_1$ can be moved into $M_{k,\tau}$. This follows immediately from Lemma \ref{Omax} and a lemma of Shiota \cite[Lemma 5.6]{W}, which we formulate as follows:
\begin{lemma}[Shiota] Let $C$ and $D$ be $n\times n$ matrices such that $[C,D]-\tau\cdot 1$ has rank $1$. If $\tau\neq 0$, then there exists a polynomial $p(t)=\sum_{r=0}^{n-1} p_rt^r$, such that $(p_0,\dots,p_{n-1})$ is arbitrarily close to $0\in \cx^n$ and $C+p(D)$ is regular semisimple.\hfill $\Box$\label{Shiota}\end{lemma}
Therefore, to prove the theorem, it is enough to show any point of $N_{k,\tau}$ can be moved, using $\TAut_R\bigl(\cx \ol Q;c\bigr)$ into $N_1$, i.e. $\tr(ji)^2$ can be made $k^2\tau^2$.
 We compute the effect on $\tr(ji)^2$ of  the strictly triangular automorphism $T(s)=\Lambda(f)\in \Aut_R\cx \ol Q$ corresponding to $f(a,b)=sab$, $s\in \cx$. Using \eqref{action}, the term linear in $s$ in $\tr(ji)^2$ is
\begin{equation} 2(j_1i_1-j_2i_2)j_2Ai_1+2j_2i_1(j_2Ai_2-j_1Ai_1),\label{ssss}\end{equation}
while the  quadratic term is
\begin{equation} 2(j_2Ai_1)^2-2j_2i_1j_2A^2i_1.\label{sss}\end{equation}
Thus, we shall be done, once we show that any point of $N_{k,\tau}$ can be transformed to a point for which either \eqref{ssss} or \eqref{sss} is nonzero (excepting, perhaps, the points for which $\tr(ji)^2$ already equals $k^2\tau^2$). If  the rank of $[A,B]-\tau\cdot 1$ is not $1$ (otherwise we are already in $N_1$), then we can use the action of $GL(k,\cx)\times GL(2,\cx)$ to assume that 
$$  i^T=\begin{pmatrix}1 & 0 &\dots & 0\\ 0 & 1 &\dots & 0 \end{pmatrix},$$
and 
$$j=\begin{pmatrix}\alpha_1 & 0 &\dots & 0\\ 0 & \alpha_2 &\dots & 0 \end{pmatrix}, \enskip \text{if $O$ is an orbit of \eqref{albe1}},$$
or
$$j=\begin{pmatrix}\alpha & 1 &0 &\dots & 0\\ 0 & \alpha & 0 &\dots & 0 \end{pmatrix}, \enskip \text{if $O$ is an orbit of \eqref{albe2}}.$$
In both cases $j_2i_1=0$ and  \eqref{sss} reduces to $(j_2Ai_1)^2$, i.e. $(A_{21}\alpha_2)^2$ in the first case or $(A_{21}\alpha)^2$ in the second case. Recall that $\alpha\neq 0$, while  $\alpha_2=0$ implies that our point  already satisfies $\tr(ji)^2=k^2\tau^2$. Thus,  \eqref{sss} is nonzero if $A_{21}\neq 0$. We can  act by  $\Aut\bigl(\cx\langle a,a^\ast\rangle ; [a,a^\ast]\bigr)$ without changing $i$ and $j$, and, so, we are done, unless the $(2,1)$-entry of $(c A+d B)^m$ is equal to zero for all $c,d\in \cx$ and $m\in \oN$. Let us assume that this last condition holds. Let $V=e_2^\perp$ (i.e. $V$ consists of vectors, the second coordinate of which is zero). Thus, $(c A+d B)^m e_1\in V$ for  all $c,d,m$. 
\begin{lemma} Let $U$ be the smallest subspace containing $e_1$ and invariant under $A$ and $B$.  Then $U\subset V$.\label{invariant}\end{lemma}
\begin{proof} We need to show that 
\begin{equation}A^{i_1}B^{j_1}\dots A^{i_n}B^{j_n}e_1\in V\label{in V_s}\end{equation}
for any $i_1,j_1,\dots, i_n,j_n\geq 0$.
We prove it by induction on the word length $m=i_1+j_1+\dots +i_n +j_n$. Suppose that \eqref{in V_s} holds for all words of length less than $m$. Since $(cA+d B)^m e_1\in V$ for any $c,d$, it is enough to show that
$ w_1[A,B]w_2 e_1\in V$
for all pairs of words $w_1=A^{p_1}B^{q_1}\dots A^{p_v}B^{q_v}$, $w_2=A^{r_1}B^{s_1}\dots A^{r_w}B^{s_w}$, the sum of lenghts of which is $m-2$. Since $[A,B]=ij+\tau\cdot 1$, we have
$$ w_1[A,B]w_2e_1= \tau w_1w_2e_1 +w_1ij w_2 e_1.$$
The first term is in $V$ owing to the inductive assumption. For the second term, $w_1ij w_2 e_1$, we know that $w_2e_1\in V$ thanks to the inductive assumption. We now observe that, with our choices of $i$ and $j$, $ij V=\cx e_1$, and, hence, $w_1ijw_2 e_1\in w_1(\cx e_1)=\cx w_1 e_1$, which is contained in $V$, again due to the inductive assumption.
\end{proof}
Therefore, we have a proper nontrivial subspace $U$, invariant for both $A$ and $B$, and, so,  $A,B$ can be simultaneously conjugated to \begin{equation} A=\begin{pmatrix} A_1 & A_3\\ 0 & A_2\end{pmatrix}, \quad B=\begin{pmatrix} B_1 & B_3\\ 0 & B_2\end{pmatrix}.\label{block}\end{equation}
Since $[A,B]-\tau\cdot 1$ has rank $2$ and $\tau\neq 0$, $[A_s,B_s]-\tau\cdot 1$ has rank $1$ for $s=1,2$.
\begin{lemma} In the case \eqref{albe1} we have $\alpha_1=-\tau\dim U$, $\alpha_2=-\tau\dim \cx^k/U$, and, in the case \eqref{albe2}, we have $\alpha =-\tau\dim U$.\label{dim} \end{lemma}
\begin{proof} (cf. the proof of Lemma 1.3 in \cite{W}) Any subspace $U$, invariant for both $A$ and $B$, is also invariant for $[A,B]-\tau\cdot 1$, and, hence $\tr (ij)_{|U}=-\tau \dim U$. We have $U\subset V$, so $\tr (ij)_{|U}=\alpha_1$ or $\tr (ij)_{|U}=\alpha$, depending on the type of orbit.  Since $\alpha_1+\alpha_2=-k\tau$, the result follows.\end{proof}

\begin{lemma} If $ij$ is in the orbit of \eqref{albe1} and $m=[A,B,i,j]\not\in \TAut_R\bigl(\cx \ol Q;c\bigr) M_{k,\tau}$, then $\alpha_1=\alpha_2$.\end{lemma}
\begin{proof}
We can assume that $A,B$ are of the form \ref{block}. Owing to the last lemma, it is enough to show that the diagonal blocks have equal size. We consider the spectral curves
$$ S_i=\left\{[\zeta_1,\zeta_2,\zeta_3]\in \oP^2;\enskip \det(\zeta_3-A_i\zeta_1-B_i\zeta_2)=0\right\},\quad i=1,2.$$
We can apply above-mentioned transitivity result of Berest and Wilson to $(A_1,B_1)$ and assume that $S_1$ is irreducible. Using the Shiota Lemma \ref{Shiota}, we can find a polynomial $p(t)$ such that $A_2+p(B_2)$ is regular semisimple.   Since $p(t)$ is arbitrarily small, $S_1$ can be assumed to remain irreducible. Thus, we can assume that $S_1$ is irreducible and $S_2$ is reduced. It follows that, 
unless $S_1=S_2$, we can use the action of $ASL(2,\cx)$ to make $A$ semisimple, in which case Lemma \ref{Omax} yields a contradiction. If $S_1=S_2$, then the diagonal blocks have the same size.
\end{proof}
\begin{remark} At this stage, we have proved Theorem \ref{transitive} for odd $k$.\end{remark}
We now reduce the question to the case \eqref{albe1}.
\begin{lemma} If $ij$ is in the orbit of \eqref{albe2} and $m=[A,B,i,j]\not\in \TAut_R\bigl(\cx \ol Q;c\bigr) M_{k,\tau}$, then there exists an $m^\prime= [A^\prime,B^\prime,i^\prime,j^\prime]\not\in \TAut_R\bigl(\cx \ol Q;c\bigr) M_{k,\tau}$ with $i^\prime j^\prime$ in the orbit of \eqref{albe1}. Moreover, if $A,B$ are of the form \eqref{block}, then so are $A^\prime,B^\prime$ with  $A_1^\prime=A_1,A_2^\prime=A_2,B_1^\prime=B_1,B_2^\prime=B_2$ and $A_3^\prime=B_3^\prime=0$. \label{onecase}\end{lemma}
\begin{proof} We can assume that $A,B$ are of the form \ref{block}, with blocks of size $n=k/2$. We also know that $i_1\in U$ and $j_2U=0$. It follows that the orbit of $(A,B,i,j)$ under the action of
$$ \begin{pmatrix} \epsilon I_{n\times n} & 0\\ 0 & \epsilon^{-1} I_{n\times n}\end{pmatrix}\times \begin{pmatrix}  \epsilon & 0\\ 0 & \epsilon^{-1}\end{pmatrix}\in GL(k,\cx)\times GL(2,\cx)$$
induces a convergent sequence $m_\epsilon=[A_\epsilon,B_\epsilon,i_\epsilon,j_\epsilon]$ in $N_{k,\tau}$ as $\epsilon\rightarrow 0$, and the limit  $m^\prime=[A^\prime,B^\prime,i^\prime,j^\prime]$ has $i^\prime j^\prime$ in the orbit of \eqref{albe1} with 
$\alpha_1=\alpha_2$. Suppose now that there is a $g\in \TAut_R\bigl(\cx \ol Q;c\bigr)$ such that $gm^\prime\in M_{k,\tau}$. This means that $\tr\bigl(g(j^\prime)g(i^\prime)\bigr)^2=k^2\tau^2$. On the other hand, we know from the previous considerations that, given the assumption on $m$, $\TAut_R\bigl(\cx \ol Q;c\bigr)$ does not change $\tr(ji)^2=(k/2)^2\tau^2$. Thus, $\tr\bigl(g(j_\epsilon)g(i_\epsilon)\bigr)^2=k^2\tau^2/4$ for every $\epsilon$, and we obtain a contradiction.
\end{proof}

Thus, we can assume that we are in the case \eqref{albe1} with $\alpha_1=\alpha_2=-k\tau/2$. As before, we use the action of $GL(k,\cx)\times GL(2,\cx)$ to assume that  
$$  i^T=\begin{pmatrix}1 & 0 &\dots & 0\\ 0 & 1 &\dots & 0 \end{pmatrix},\quad
j=\begin{pmatrix}-k\tau/2 & 0 &\dots & 0\\ 0 & -k\tau/2 &\dots & 0 \end{pmatrix}.$$
Observe now that the action of $GL(2,\cx)$ leaves $ji$ invariant, and, hence, preserves the condition $j_2i_1=0$. We can repeat the arguments after Lemma \ref{Shiota} for any $(i,j)$ in the orbit of $GL(2,\cx)$ and conclude that, if our point cannot be moved into $M_{k,\tau}$, then:
$$ \left((cA+dB)^m\right)_{21}= \left((cA+dB)^m\right)_{12}=0,\quad \left((cA+dB)^m\right)_{11}= \left((cA+dB)^m\right)_{22}$$
for all $c,d\in \cx$ and $m\in \oN$. This, in turn, implies that if we put $v=pe_1+qe_2$, $w=qe_1-pe_2$, then $w^T(cA+dB)^m v=0$ for every $c,d\in \cx$ and $m\in \oN$. We can now repeat the proof of Lemma \ref{invariant}, and conclude that, for every $\zeta\in [p,q]\in \oP^1$, there is an $(A,B)$-invariant subspace $U_\zeta$ containing $v=pe_1+qe_2$ and annihilated by $w^T=(qe_1-pe_2)^T$. The argument used in the proof of Lemma \ref{dim} shows that $U_\zeta\cap U_{\zeta^\prime}=0$, if $\zeta\neq \zeta^\prime$.
\par
Using the decomposition $\cx^{2n}=U_0\oplus U_\infty$, we write $A$ and $B$ as
$$A=\begin{pmatrix} A_1 & 0\\ 0 & A_2\end{pmatrix}, \quad B=\begin{pmatrix} B_1 & 0\\ 0 & B_2\end{pmatrix}.$$
Using the Shiota lemma, we can assume that both $A_1$ and $B_1$ are regular semisimple. We then have:
\begin{lemma} In the above situation,  $A_2=A_1$ and $B_2=B_1$.\label{equal}\end{lemma} 
We begin the proof by showing that $A_2$ has the same eigenvalues as $A_1$ and similarly for $B_2$ and $B_1$:
\begin{lemma} Let $\cx^{2n}=U_0\oplus U_1$,  $\dim U_0=\dim U_1=n$ and let $X:\cx^{2n}\to \cx^{2n}$ preserve both $U_0$ and $U_1$. Suppose also that there exists a further $n$-dimensional $X$-invariant subspace $V$, such that $V\cap U_0=V\cap U_{1}=0$ and $X_{|V}$ is regular semisimple with eigenvalues $\lambda_1,\dots,\lambda_n$. Then $\cx^{2n}=\bigoplus_{i=1}^n W_i$, where $\dim W_i=2$, $X_{|W_i}=\lambda_i\cdot 1$ and $\dim W_i\cap U_j=1$ for any $i=1,\dots,n$, $j=0,1$.\label{lala}\end{lemma}
\begin{proof} Let $v_i\in V$ be an eigenvector for $X$, $Xv=\lambda _iv_i$. Decompose $v=v_i^0+v_i^1$, with $v_i^j\in U_j$, $v_i^j\neq 0$. Since $Xv_i^j\in U_j$ and $\lambda_i v_i =\lambda_i v_i^0+\lambda_i v_i^1$, we must have $Xv_i^j=\lambda v_i^j$, $j=0,1$. Setting $W_i=\Span\{v_i^0,v_i^1\}$ gives the required decomposition.
\end{proof}
We now continue with the proof of Lemma \ref{equal}. We can conjugate $A$, so that $A_1=A_2$ is diagonal with distinct eigenvalues $\lambda_1,\dots,\lambda_n$. It follows that (up to a non-zero multiple):
\begin{equation}i^T=\begin{pmatrix} 1 &\dots & 1&0 & \dots & 0\\0 & \dots & 0 &1 &\dots & 1\end{pmatrix},\quad j=\begin{pmatrix}-\tau &\dots & -\tau&0 & \dots & 0\\ 0 & \dots & 0 &-\tau &\dots & -\tau\end{pmatrix}.\label{i,j}\end{equation}
Since, for any $[p,q]\in \oP^1$, $p i_1+qi_2\in U_{[p,q]}$, the last lemma implies  that $U_{[p,q]}$ is generated by $\{pe_i+qe_{n+i};\;i=1,\dots,n\}$. Let $v_i=\sum_{j=1}^n \alpha_{ij} e_j$, $w_i=\sum_{j=1}^n \beta_{ij} e_{n+j}$ be the eigenvectors of $B_1$ and $B_2$, respectively (with the same eigenvalue). Since both $B_1$ and $B_2$ are regular semisimple, we can apply to them the arguments already used for $A_1$ and $A_2$ and conclude that $U_{[p,q]}$ is generated by $\{pv_i+qw_i;\;i=1,\dots,n\}$.
But $$pv_i+qw_i=\sum\alpha_{ij}(pe_i+qe_{n+i}) +\sum (\beta_{ij}-\alpha_{ij})qf_i,$$ and, hence, $\sum (\beta_{ij}-\alpha_{ij})qf_i\in U_\infty\cap U_{[p,q]}$. Therefore $\beta_{ij}=\alpha_{ij}$ for all $i,j$ and Lemma \ref{equal} is proved.

Thus, we can assume that $A$ and $B$ can be simultaneously conjugated to 
$$ A=\begin{pmatrix} \Lambda & 0\\ 0 & \Lambda \end{pmatrix}, \quad B=\begin{pmatrix} D & 0\\ 0 & D\end{pmatrix},$$
where $\Lambda=\diag(\lambda_1,\dots,\lambda_n)$, $D$ is regular semisimple, and  $[\Lambda,D]-\tau\cdot 1$ has rank $1$.
Moreover $i,j$ are of the form \eqref{i,j}.
We now act by the strictly triangular automorphism \eqref{action} with $p(a)=a$. This leaves invariant $A,i_1,j_2$ and changes $B$ to 
$$ B^\prime=\begin{pmatrix} D & -\tau E\\ 0 & D\end{pmatrix},$$
where every entry of $E$ is $1$. It also sends $i_2$ and $j_1$ to:
$$ i_2^\prime=(\lambda_1,\dots,\lambda_n,1,\dots,1)^T,\quad j_1^\prime=\tau(-1,\dots,-1,\lambda_1,\dots,\lambda_n).$$
We now act by an element of $SL(2,\cx)\times GL(2,\cx)$ to replace $(A,B^\prime)$ by $(-B^\prime, A)$ and $i_1,j_2$ by $i_2^\prime, j_1^\prime$. We act again by \eqref{action} with $p(a)=a$ and end up with the pair of matrices 
$ A^{\prime\prime}=-B^\prime,\; B^{\prime\prime}=A+i_2^\prime j_1^\prime $. We know from Lemmas \ref{invariant} and \ref{dim} that $ A^{\prime\prime}$ and $ B^{\prime\prime}$ have a common $n$-dimensional subspace $U$. Moreover, Lemmas \ref{onecase} and \ref{lala} imply that $U$ is generated by eigenvectors of $-B$ (one for every eigenvalue).  The only subspace invariant for 
$ A^{\prime\prime}$, which satisfies the latter condition is the one spanned by the first $n$ coordinate vectors, but this subspace is not invariant for  $B^{\prime\prime}$. This contradiction establishes Theorem \ref{transitive}.

\begin{ack} This paper was mostly written in 2007, when the first author was a Humboldt Fellow at the University of G\"ottingen. The Fellowship and the hospitality of the host university are gratefully acknowledged. 
\par
Both authors thank George Wilson for comments and  William Crawley-Boevey for explaining how complete integrability of instanton moduli spaces fits in with general results on quiver varieties and for the reference \cite{Bock}. 
\end{ack}

\end{document}